\documentclass{amsart}

\usepackage{amssymb,amsmath,color}
\usepackage{amsthm}
\usepackage{url}
\usepackage{tikz}
\usepackage{hyperref}
\usepackage{ulem}
\usepackage{soul}

\definecolor{red}{rgb}{1,0,0}
\definecolor{blue}{rgb}{.2,.2,.8}

\def\dsf{\mathrm{dsf}}

\def\st{\mathrm{st}}

\def\sp{\mathrm{sp}}
\def\sa{\mathrm{sa}}
\def\SA{\mathcal{SA}}
\def\pa{\mathrm{pa}}
\def\spa{\mathrm{sp{\small a}}}

\def\rep{\mathrm{rep}}
\def\add{\mathrm{add}}
\def\bin{\mathrm{bin}}
\def\val{\mathrm{val}}
\def\rb{\mathrm{rb}}
\def\OB{\mathcal{OB}}
\def\B{\mathcal{B}}
\def\P{\mathcal{P}}
\def\OC{\mathcal{OC}}
\def\sf{\mathrm{sf}}
\def\SF{\mathcal{SF}}

\def\DSF{\mathcal{DSF}}
\def\ST{\mathcal{ST}}
\def\sl{\mathrm{sl}}
\def\SL{\mathcal{SL}}
\def\sb{\mathrm{sb}}

\def\SPa{\mathcal{SP}a}
\def\SP{\mathcal{SP}}
\def\SB{\mathcal{SB}}
\def\nc{\mathrm{nc}}
\def\snc{\mathrm{snc}}
\def\SNc{\mathcal{SN}c}
\def\sc{\mathrm{snc}}
\def\SC{\mathcal{SN}c}
\def\HB{\mathcal{HB}}
\def\psf{\mathrm{psf}}
\def\pdsf{\mathrm{pdsf}}
\def\psl{\mathrm{psl}}
\def\pst{\mathrm{pst}}
\def\pspa{\mathrm{pspa}}
\def\psa{\mathrm{psa}}
\def\psp{\mathrm{psp}}
\def\psnc{\mathrm{psnc}}
\def\psb{\mathrm{psb}}
\def\phb{\mathrm{phb}}
\def\OBST{\overline{\OB}^{\dag}_{\overline{\mathrm R}}}

\def\i4{\ \ \ \ }
\def\h5{\ \ \ \ \ }
\def\R{\mathrm R}

\newtheorem{theorem}{Theorem}[section]
\newtheorem{corollary}[theorem]{Corollary}
\newtheorem{proposition}{Proposition}

\newtheorem{lemma}[theorem]{Lemma}

\theoremstyle{definition}
\newtheorem{definition}{Definition}
\newtheorem{example}{Example}
\newtheorem{remark}{Remark}

\begin{document}

\allowdisplaybreaks

\title{Partitions enumerated by self-similar sequences}
\author{Cristina Ballantine}\address{Department of Mathematics and Computer Science\\ College of the Holy Cross \\ Worcester, MA 01610, USA \\} 
\email{cballant@holycross.edu} 
\author{George Beck} \address{Department of Mathematics and Statistics\\ Dalhousie University\\ Halifax, NS, B3H 4R2, Canada \\} \email{george.beck@gmail.com}

\maketitle

\begin{abstract}

 The Fibonacci numbers are the prototypical example of a recursive sequence, but grow too quickly to enumerate sets of integer partitions. The same is true for the other classical sequences $a(n)$ defined by Fibonacci-like recursions: the tribonacci, Padovan, Pell, Narayana's cows, and Lucas sequences. For each sequence $a(n)$, however, we can define a  related sequence $\sa(n)$ by   defining $\sa(n)$ to have the same recurrence and initial conditions as $a(n)$, except that $\sa(2n)=\sa(n)$. Growth is no longer a problem: for each $n$ we construct recursively a set $\SA(n)$ of partitions of $n$ such that the cardinality of $\SA(n)$ is $\sa(n)$. We study the properties of partitions in $\SA(n)$ and in each case we give non-recursive descriptions. We  find congruences for $\sa(n)$ and also for $\psa(n)$, the  total number of parts in all partitions in $\SA(n)$. 
\end{abstract}

\section{Introduction} An integer partition of a positive integer $n$ is a weakly decreasing sequence of positive integers whose sum is $n$. The positive integers in the sequence are called parts. The set of partitions of $n$
 is denoted by $\P(n)$ and the number of partitions of $n$
 is denoted by $p(n)$.
For example, the five partitions of $4$ are $(4)$, $(3,1)$, $(2,2)$, $(2,1,1)$, and $(1,1,1,1)$, so $p(n) = 5$. Since the empty partition is the only partition of $0$, we have $p(0)=1$.

As usual, $\{F_n\}_{n\geq 0}$ denotes the sequence of Fibonacci numbers. For all $n\geq 9$, $p(n)<F_n$ (see for example \cite{AE}). Thus, the numbers $F_n$ grow too fast to enumerate a subset of $\P(n)$. However, one can define a related integer sequence that enumerates the set $\OB(n)$ of odd binary partitions of $n$, which are the partitions of $n$ whose parts are powers of $2$, each part occurring an odd number of times. 

\begin{definition} \label{def-sf} The semi-Fibonacci sequence, $\{\sf(n)\}_{n\geq 0}$ is defined recursively by $\sf(0)=1$, $\sf(1)=1$, and for $n\geq 1$, \begin{align*}\sf(2n)&=\sf(n)\\ \sf(2n+1)&=\sf(2n)+\sf(2n-1).\end{align*} 
\end{definition}

The following table shows the values of the sequence $\sf(n)$ for $n = 0, 1, \ldots, 12$.

$$\begin {array}
{c | ccccccccccccc}
n & 0 & 1 & 2 & 3 & 4 & 5 & 6 & 7 & 8 & 9 & 10 & 11 & 12 \\
\hline
\sf(n) & 1 & 1 & 1 & 2 & 1 & 3 & 2 & 5 & 1 & 6 & 3 & 9 & 2
\end {array} $$

Andrews proved in \cite{A19} that $|\OB(n)|=\sf(n)$. In \cite{Beck}, the second author defined the set $\SF(n)$ of semi-Fibonacci partitions of $n$ by 
$\SF(1):=\{(1)\}$, $\SF(2):=\{(2)\}$ 
and, for $n\geq 1$, 
the partitions in $\SF(2n)$ are the partitions of $\SF(n)$ with each part doubled, and 
the partitions in $\SF(2n+1)$ are the partitions of $\SF(2n)$ each with an additional part equal to $1$, or the partitions in $\SF(2n-1)$ each with the unique odd part increased by $2$.
Then, $|\SF(n)|=\sf(n)$.

A non-recursive description of $\SF(n)$ was given in \cite[Theorem 4]{AMN}: for a positive integer $m$, the $2$-adic valuation of $m$, denoted $\val_2(m)$, is the non-negative integer $k$ such that $2^k \mid m$ but $2^{k+1} \nmid m$. Then, for $n\geq 1$, $\SF(n)$ is the set of partitions of $n$ with distinct $2$-adic valuations of parts.

Given an integer sequence $a(n)$ defined by a set of initial conditions and a linear recurrence, there is an associated integer sequence $\sa(n)$ defined with the same initial conditions and the same recurrence for odd $n$, but with $\sa(2n) = \sa(n)$. In the literature, such a sequence is called self-similar \cite{selfsimilar} because the subsequence of even terms is the same as the entire sequence.  

Motivated by semi-Fibonacci partitions, for each such sequence $\sa(n)$, we can define recursively a set of partitions $\SA(n)$ of $n$ enumerated by $\sa(n)$. 
These definitions can be varied; we have chosen definitions that allow us to also give   non-recursive descriptions of the partitions in $\SA(n)$.

We discuss parity properties for most of the self-similar sequences we introduce. In many cases we investigate properties and congruences for the total number of parts in all partitions in $\SA(n)$ for the respective sets of partitions.

The paper is organized as follows. In section \ref{prelim}, we discuss necessary background on partitions and introduce notation used throughout the paper. In section \ref{sec_gen} we determine the generating function for a self-similar sequence defined by a recurrence of order at most three at  odd arguments. 
In the subsequent sections, we study different self-similar sequences and related partitions as follows. 
\begin{center}
\begin{tabular}{c|l} section & sequence\\ 
\hline \ref{sec_st} & semi-tribonacci\\ \ref{sec_spa} & semi-Padovan\\ \ref{sec_sp} & semi-Pell\\ \ref{sec_snc} & semi-Narayana's cows \\ \ref{sec_dsf} & delayed semi-Fibonacci\\ \ref{sec_sl} & semi-Lucas \\ \ref{sec_sb} &  Stern-Brocot
\end{tabular}
\end{center}
For the convenience of the reader, we summarize notation and definitions in section \ref{sec_notation}. We offer some concluding remarks and open problems in section \ref{sec_conclusions}.

\section{Preliminaries and notation}
\label{prelim}

If $\lambda$ is a partition of $n$, we write $\lambda\vdash n$. The weight (or size) of a partition $\lambda$, denoted by $|\lambda|$, is the sum of its parts. Thus, if $\lambda \vdash n$, then $|\lambda|=n$. The length of a partition $\lambda$, denoted by $\ell(\lambda)$, is the number of parts of $\lambda$. We write a partition as a list, a formal sum, or a concatenation when unambiguous. For example, the partition $\lambda=(4,4,3,2,2,1)$ of $16$ can also be written as $4+4+3+2+2+1$ or $443221$. We denote the empty partition by $(\ )$. We also view a partition as the multiset of its parts and write $a\in \lambda$ to mean that $a$ is a part of $\lambda$. The number of times an integer $i$ occurs as a part in the partition $\lambda$ is denoted by $m_\lambda(i)$ and is referred to as the multiplicity of $i$ in $\lambda$. For example, if $\lambda=(4,4,3,2,2,1)$, then $m_\lambda(4)=2$. 
For more on the theory of partitions, we refer the reader to \cite{Andrews98}. 

An overpartition of $n$ is a partition of $n$ in which the first occurrence of a part may be overlined \cite{Corteel}. For example, $(4,4,3,2,2,1), (\overline 4,4,3,2,2,1)$, and $(\overline 4,4,\overline 3 ,2,2,1)$ are some of the overpartitions of $16$.

We define several operations on partitions. 

 If $\lambda$ and $\mu$ are partitions, then $\lambda\sqcup \mu$ denotes the partition whose multiset of parts is the union of the multisets of parts of $\lambda$ and $\mu$. For example, if $\lambda=(2,1,1)$ and $\mu=(3,2,1,1)$, we have $\lambda\sqcup\mu=(3,2,2,1,1,1,1)$. Similarly, if $\mathcal S$ is a set of partitions, define ${\mathcal S \sqcup \mu =\{\lambda \sqcup \mu \mid \lambda \in \mathcal S\}}$. If $\mathcal S=\emptyset$, then $\mathcal S\sqcup \lambda=\emptyset$.

 If each part of $\mu$ occurs in $\lambda$ with equal or greater multiplicity, we say that $\mu$ is contained in $\lambda$, or a subpartition of $\lambda$, and write $\mu\subseteq \lambda$. If $\mu\subseteq \lambda$, we denote by $\lambda \setminus \mu$ the partition obtained from $\lambda$ by removing the multiset of parts of $\mu$. For example, if $\lambda=(7,4,4,4,2,1,1)$ and $\mu=(4,2,1,1)$, then $\lambda\setminus \mu=(7,4,4)$.

Let $m$ be a positive integer. For a partition $\lambda=(\lambda_1, \lambda_2, \ldots, \lambda_\ell)$, we define $m\lambda:=(m\lambda_1, m \lambda_2, \ldots, m\lambda_\ell)$. Moreover, if every part of $\lambda$ is divisible by $m$, we define $\displaystyle \lambda/m:=(\lambda_1/m, \lambda_2/m, \ldots, \lambda_\ell/m)$. Occasionally, we also write $\frac{1}{m}\lambda$ for $\lambda/m$.

For a partition $\lambda=(\lambda_1, \lambda_2, \ldots, \lambda_\ell)$, we define $\add_{1,1}(\lambda):=(\lambda_1+1, \lambda_2, \ldots, \lambda_\ell)$, the partition obtained by adding $1$ to the first part of $\lambda$. Vacuously, $\add_{1,1}((\ ))=(\ )$. If  $\lambda$ has an odd part, we denote by $\lambda^{o+2}$ the partition obtained from $\lambda$ by adding $2$ to the largest odd part. 

Again, we extend this notation to sets.  Given a set of partitions $\mathcal S$,  define ${m \mathcal S =\{m \lambda\mid\lambda \in \mathcal S\}}$, ${\mathcal S^{o+2}=\{ \lambda^{o+2}\mid\lambda \in \mathcal S\}}$ (assuming each partition in $\mathcal S$ has an odd part), and $\add_{1,1}(\mathcal S)=\{\add_{1,1}(\lambda) \mid \lambda\in \mathcal S\}$.

A binary partition has only powers of $2$ for parts. We denote the set of binary partitions of $n$ by $\B(n)$. 

As mentioned in the previous section, $\OB(n)$ denotes the subset of $\B(n)$ consisting of partition in which all parts have odd multiplicities. 
We denote by $\overline\OB(n)$ the set of  odd binary overpartitions of $n$; removing the overlines in $\lambda \in \overline\OB(n)$ gives  an odd binary partition. For example, the binary overpartitions of $3$ are $2+1, \, \overline{2}+1,\,  2+\overline 1,\, \overline{2}+\overline{1},\,  1+1+1,\,  \overline 1+1+1. $

Throughout the paper, if $\mathfrak  f$ is a function on a set of partitions, we omit double parentheses and write $\mathfrak  f(\lambda_1, \lambda_2, \ldots, \lambda_\ell)$ instead of $\mathfrak  f((\lambda_1, \lambda_2, \ldots, \lambda_\ell))$.

 We define the weight-preserving function $\rep_2: \P(n) \to \B(n)$ as follows. If $n\geq 1$ and $\val_2(n)=k$, define $\rep_2(n)$ to be the partition with $n/2^k$ parts all equal to $2^k$. Then define $$\rep_2(\lambda_1, \lambda_2, \ldots, \lambda_\ell):=\bigsqcup_{i=1}^\ell \rep_2(\lambda_i).$$
 For example, $\rep_2(12,6,1,1,1)=(4,4,4,2,2,2,1,1,1)$. Note that $\rep_2$ is not injective. For example, $\rep_2(12,6,2,1,1,1)=(4,4,4,2,2,2,1,1,1)=\rep_2(12,6,3,2).$
 
As mentioned above, in \cite[Theorem 4]{AMN} it is shown that, for $n\geq 1$, $\SF(n)$ is the set of partitions of $n$ with distinct $2$-adic valuations of parts. Then, $\rep_2$ gives a bijection between $\SF(n)$ and $\OB(n)$. Here, the inverse of $\rep_2$ takes a partition in $\OB(n)$ and merges all equal parts into a single part. For example, $\rep_2^{-1}(4,4,4,2,2,2,1,1,1)=(12,6,3)$.

 \begin{example} In the table below, the odd binary partitions in $\OB(n)$ are, in order, listed as $\rep_2(\lambda)$ for the partitions $\lambda \in \SF(n)$. 
$$\begin{array}{c|c|c} n& \SF(n) & \OB(n)\\ \hline 1 & 1 & 1 \\ 2 & 2& 2 \\ 3 & 2+1, 3 & 2+1, 1+1+1 \\ 4& 4 & 4\\ 5 & 4+1, 3+2, 5 & 4+1, 2+1+1+1, 1+1+1+1+1 \end{array} $$
\end{example}
 
 We define the weight-preserving function $\bin: \P(n)\to \B(n)$ as follows. First, for $n\geq 1$, $\bin(n)$ is the unique partition of $n$ into distinct powers of $2$. Then define $$\bin(\lambda_1, \lambda_2, \ldots, \lambda_\ell):=\bigsqcup_{i=1}^\ell \bin(\lambda_i).$$ For example, $\bin(7, 2) = (4, 2, 2, 1)$. Note that $\bin$ is not injective on the set of partitions of $n$. For example, $\bin(6,3)=(4, 2, 2, 1)=\bin(7,2)$.

The next theorem emphasizes the difference between $\rep_2$ and $\bin$. 

\begin{theorem} The partition $\lambda$ is binary if and only if $\bin(\lambda)= \rep_ 2(\lambda)$.
\end{theorem}
\begin{proof}

If $n$ is a power of 2, $\bin(n)=(n)=\rep_2(n)$.

So if $\lambda$ is a binary partition, \begin{equation}\label{b-r}\bin(\lambda)=\lambda=\rep_2(\lambda).\end{equation}

If $\lambda$ is not a binary partition, let $\beta$ be the subpartition of $\lambda$ consisting of the parts that are powers of 2 and let $\gamma=\lambda\setminus \beta$. By \eqref{b-r}, to prove $\bin(\lambda)\neq \rep_2(\lambda)$, it is sufficient to prove that $ \bin(\gamma)\neq \rep_ 2(\gamma)$.

Let $m$ be the largest power of 2 that divides all the parts of $\gamma$. Since the parts of $\gamma$ are not powers of $2$, it follows that $m$ is not a part of $\gamma$. The partition $\gamma/m$ has at least one odd part and all of its odd parts of are greater than 1. The number of 1s in $\rep_2(\gamma/m)$ is the sum $S$ of the odd parts of $\gamma/m$; the number of 1s in $\bin(\gamma/m)$ is the number $N$ of odd parts of $\gamma/m$. But $S>N$, so $\rep_2(\gamma/m)\neq \bin(\gamma/m)$. Then, since $m$ is a power of $2$, we have $\rep_2(\gamma)=m\, \rep_2(\gamma/m)\neq m\, \bin(\gamma/m)=\bin(\gamma)$, which completes the proof.
\end{proof}

We aim to keep the notation intuitive throughout the paper. A self-similar sequence is denoted by one or two lowercase letters prefixed by the letter “s" (for “semi"); for example, $\sp(n)$ denotes the semi-Pell sequence and $\spa(n)$ denotes the semi-Padovan sequence. The corresponding set of partitions is denoted by the same letters in capital calligraphy style; for example, the set of partitions enumerated by $\sp(n)$ is denoted by $\SP(n)$, which are the semi-Pell partitions. (Sometimes we need to add a lowercase letter, as in $\SPa$, the semi-Padovan partitions.) For the total number of parts in all partitions in $\SP(n)$, for example, we use $\psp(n)$; that is, we the same lower case letters preceded by $\mathrm{p}$.

In general, to avoid repetition, we omit the proofs of propositions; they can be proved from some combination of using the recurrences, induction, or previous results. The omitted proofs are similar to those we present in detail for some theorems.

 \section{The generating function of a self-similar sequence} \label{sec_gen}

  In this section, we begin by considering a general self-similar sequence with a recurrence of order at most $3$ at positive odd arguments and use analytic methods similar to those in \cite{A19} and \cite{KM} to find its generating function. In subsequent sections, specializing that recurrence gives combinatorial interpretations for self-similar sequences based on well-known Fibonacci-like sequences.

Consider the sequence $\{f(n)\}_{n\geq 0}$ defined recursively by: $f(0)=a_0, f(1)=a_1$, and, for $n\geq 1$, by
\begin{align*}f(2n)& = f(n), \\ f(2n+1) & =c_1f(2n)+c_2f(2n-1)+c_3f(2n-2).
\end{align*}

Let $F(x)$ be the generating function for $\{f(n)\}_{n\geq 0}$. Thus, $$F(x)=\sum_{n\geq 0}f(n)x^n.$$ 
We have $$\sum_{n\geq0}f(2n)x^{2n}= a_0+\sum_{n\geq1}f(2n)x^{2n} =a_0+ \sum_{n\geq 1}f(n)x^{2n} = \sum_{n\geq0}f(n)x^{2n}=F(x^2),$$ and 
\begin{alignat*}{3}\sum_{n\geq 0}f(2n+1)x^{2n+1} & = a_1x&& + \sum_{n\geq 1}(c_1f(2n)+c_2f(2n-1)+c_3f(2n-2))x^{2n+1}\\ 
& = a_1x && + c_1x\sum_{n\geq1}f(2n)x^{2n} +c_2x^2\sum_{n\geq 1}f(2n-1)x^{2n-1} \\ & && +c_3x^3\sum_{n\geq 1}f(2(n-1))x^{2(n-1)} \\
& = a_1x&& + c_1x\sum_{n\geq1}f(2n)x^{2n} +c_2x^2\sum_{n\geq 0}f(2n+1)x^{2n+1} \\ & && +c_3x^3\sum_{n\geq 0}f(2n)x^{2n} .%\\
%& =a_1x+c_1x(F(x^2)-a_0)+c_2x^2(F(x)-F(x^2))+c_3x^3F(x^2).
\end{alignat*}

Thus, $$\sum_{n\geq 0}f(2n+1)x^{2n+1} = \frac{1}{1-c_2x^2}\left(a_1x+c_1x(F(x^2)-a_0)+c_3x^3F(x^2)\right).$$
Then
\begin{align*}F(x) & = \sum_{n\geq 0}f(2n)x^{2n}+\sum_{n\geq 0}f(2n+1)x^{2n+1}\\ 
& =F(x^2)+ \frac{1}{1-c_2x^2}\left(a_1x+c_1x(F(x^2)-a_0)+c_3x^3F(x^2)\right)\\& = \frac{(a_1-a_0c_1)x}{1-c_2x^2}+\left(1+\frac{c_1x+c_3x^3}{1-c_2x^2}\right)F(x^2).
\end{align*}
Let $$A(x):=\frac{(a_1-a_0c_1)x}{1-c_2x^2} \text{\ and \ }B(x):=1+\frac{c_1x+c_3x^3}{1-c_2x^2}.$$ Then \begin{equation} F(x)=A(x)+B(x)F(x^2).\label{gf} \end{equation}
Iterating \eqref{gf}, leads to \begin{equation} F(x)=\sum_{i=0}^k A(x^{2^i})\prod_{r=0}^{i-1}B(x^{2^r})+ F(x^{2^{k+1}})\prod_{r=0}^{k}B(x^{2^r}), \label{gf_it} \end{equation} 
with the convention that an empty product such as $\prod_{r=0}^{-1}B(x^{2^r})$ is taken to be $1$.

In \eqref{gf_it}, we let $k\to \infty$. Using $\displaystyle \lim_{k\to \infty}F(x^{2^{k+1}})=a_0$ we get \begin{equation} \label{gf_exp} F(x)=\sum_{i=0}^\infty A(x^{2^i})\prod_{r=0}^{i-1}B(x^{2^r})+ a_0\prod_{r=0}^\infty B(x^{2^r}).\end{equation}
It is easily checked that \eqref{gf_exp} satisfies equation \eqref{gf} and the initial coefficients agree with the initial terms of the sequence.

\noindent\underline{Case 1:} If $a_1-a_0c_1\neq0$, then $A(x^{2^i})\neq 0$ for all $i\geq 0$ and we have 

\begin{align}\label{big} F(x) = \sum_{i=0}^\infty \frac{(a_1-a_0c_1)x^{2^i}}{1-c_2x^{2^{i+1}}}&  \prod_{r=0}^{i-1} \left(1+\frac{c_1x^{2^r}+c_3x^{3\cdot 2^r}}{1-c_2x^{2^{r+1}}}\right)\\ & \notag +a_0\prod_{r=0}^{\infty} \left(1+\frac{c_1x^{2^r}+c_3x^{3\cdot 2^r}}{1-c_2x^{2^{r+1}}}\right).
\end{align}

\noindent\underline{Case 2:} If $a_1-a_0c_1=0$, then $A(x^{2^i})= 0$ for all $i\geq 0$ and we have 

\begin{align}\label{small}F(x) & = a_0\prod_{r=0}^{\infty} \left(1+\frac{c_1x^{2^r}+c_3x^{3\cdot 2^r}}{1-c_2x^{2^{r+1}}}\right).
\end{align}

 In particular, if $f(n)$ is the semi-Fibonacci sequence, we have $a_0=a_1=1$, $c_1=c_2=1$, and $c_3=0$. Thus, we are in Case 2. and \eqref{small} becomes 

$$F(x) = \prod_{r=0}^{\infty} \left(1+\frac{x^{2^r}}{1-x^{2^{r+1}}}\right)=\sum_{n=0}^\infty|\OB(n)|q^n,$$ the generating function for the sequence $|\OB(n)|$.

\section{Semi-tribonacci partitions}\label{sec_st}

\begin{definition} The semi-tribonacci sequence $\{\st(n)\}_{n\geq 0}$ is defined recursively by $\st(0)=0$, $\st(1)=1$, and for $n\geq 1$, \begin{align*}\st(2n)&=\st(n),\\ \st(2n+1)&=\st(2n)+\st(2n-1)+\st(2n-2).\end{align*} 
\end{definition}
The following table shows the values of the semi-tribonacci sequence $\st(n)$ for $n = 0, 1, \ldots, 12$.
$$\begin{array}
{c|ccccccccccccc}
n & 0 & 1 & 2 & 3 & 4 & 5 & 6 & 7 & 8 & 9 & 10 & 11 & 12 \\
\hline
\st(n) & 0 & 1 & 1 & 2 & 1 & 4 & 2 & 7 & 1 & 10 & 4 & 15 & 2 
\end{array} $$

This is the sequence $\{f(n)\}_{n\geq 0}$ of Section \ref{sec_gen} with $a_0=0, a_1=1$, and $c_1=c_2=c_3=1$. Thus, we are in Case 1 of Section \ref{sec_gen} and \eqref{big} becomes 

\begin{equation}\label{gfst} F(x) =\sum_{i=0}^\infty \frac{x^{2^i}}{1-x^{2^{i+1}}} \prod_{r=0}^{i-1} \left(1+\frac{x^{2^r}+x^{3\cdot 2^r}}{1-x^{2^{r+1}}}\right).\end{equation}
Then, $F(x)$ is the generating function for $|\OBST(n)|$,  where 
 $\OBST(n)$ is the set of odd binary overpartitions of $n$ such that only parts that appear with multiplicity at least three may be overlined and the largest part may not be overlined. For example, $\OBST(7)$ is the set\begin{align*}\{4+2+1,\ & 4+1+1+1,\, 4+\overline{1}+1+1,\, 2+2+2+1,2+1+1+1+1+1,\\ & 2+\overline{1}+1+1+1+1,\, 1+1+1+1+1+1+1 \}.\end{align*}

 To understand the combinatorial  interpretation of \eqref{gfst} above, notice that for a fixed $i\geq 0$, the  power series expansion of  $\displaystyle \frac{x^{2^i}}{1-x^{2^{i+1}}}$   keeps track of  the largest part in  a partition (counted with multiplicity). For example, $x^{5\cdot 2^i}$ contributes  five parts equal to $2^i$ to a partition.  Each factor in $\displaystyle \prod_{r=0}^{i-1}\left(1+ \frac{x^{2^r}}{1-x^{2^{r+1}}}+\frac{x^{3\cdot{2^r}}}{1-x^{2^{r+1}}}\right)$ is interpreted as follows:  the series expansion of $\displaystyle \frac{x^{2^r}}{1-x^{2^{r+1}}}$ keeps track of parts equal to $2^r$,  none being overlined, and the  the series expansion of $\displaystyle \frac{x^{3\cdot{2^r}}}{1-x^{2^{r+1}}}$ keeps track of  parts equal to $2^r$ if the first part equal to $2^r$ is overlined. In the latter case there are at least three parts equal to $2^r$.

 Thus, we have the following combinatorial interpretation for $\st(n)$.
 
 \begin{theorem} \label{st-ob} For $n\geq 0$, $|\OBST(n)|=\st(n)$ for $n\geq 0$.
\end{theorem}

\begin{remark}The odd binary partitions of $n$, which are enumerated by $\sf(n)$, are precisely the overpartitions in $\OBST(n)$ that have no overlined part. Hence, $\sf(n)\leq \st(n)$ for all $n\geq 0$. 
\end{remark}

We next define a subset of partitions of $n$ enumerated by $\st(n)$. 

\begin{definition}
The semi-tribonacci partitions are defined recursively by \\ $\ST(0)=\emptyset$, $\ST(1)=\{(1)\}$, $\ST(3)=\{(3), (2,1)\}$ and for $n\geq 1$, \begin{align*}\ST(2n)& = 2 \ST(n),\\ \ST(2n+3) & = \ST(2n+2) \sqcup (1) \bigcup \ST(2n+1)^{o+2} \bigcup \ST(2n) \sqcup (1,1,1).\end{align*}\end{definition}

 \begin{example} The table below shows the partitions in $\ST(n)$ for $n=1,\ldots,8$. 
$$\begin{array}{c|c} n& \ST(n) \\ \hline  1 & 1 \\ 2 & 2\\ 3 & 3, 2+1 \\ 4& 4 \\ 5 & 5, 3+2,4+1,2+1+1+1 \\ 6&6,4+2 \\ 7 & 7, 4+3,5+2,6+1,4+2+1,3+2+1+1,4+1+1+1 \\ 8 & 8 \end{array} $$
 
 \end{example}

\begin{remark}\label{R1} Let $\lambda\in \ST(n)$. Clearly, if $n$ is even, $\lambda$ has no odd parts. If $n$ is odd, it is easily seen by induction that $\lambda$ has one or three odd parts and in the latter case at least two of the odd parts equal $1$.
\end{remark}

The recursive definition of semi-tribonacci partitions together with Remark \ref{R1} leads to the following corollary to Theorem \ref{st-ob}.

\begin{corollary} \label{st} For $n\geq 0$, $|\ST(n)|=|\OBST(n)|.$
\end{corollary}

Next, we give a non-recursive description of the semi-tribonacci partitions that will allow us give a combinatorial proof of Corollary \ref{st}.

Given a partition $\lambda=(\lambda_1, \lambda_2, \ldots, \lambda_{\ell(\lambda)})$ let $v(\lambda)=(v_1(\lambda), v_2(\lambda), \ldots, v_{\ell(\lambda)}(\lambda))$ be the sequence of non-negative integers obtained by arranging $\val_2(\lambda_i)$, $1\leq i\leq \ell(\lambda)$, in non-increasing order. Thus, $v(\lambda)$ is the partition whose parts are the $2$-adic valuations of the parts of $\lambda$.

\begin{theorem}\label{T-st} Let $n>0$. The set $\ST(n)$ consists of all partitions $\lambda$ of $n$ satisfying the following conditions. 
\begin{itemize}\item[(a)] $v_1(\lambda)$ has multiplicity one in $v(\lambda)$ (i.e. $v_1(\lambda)>v_2(\lambda)$); \item[(b)] if $1<j\leq \ell(\lambda)$, then $v_j(\lambda)$ has multiplicity one or three;
\item[(c)]if $v_j(\lambda)$ has multiplicity three, then $\lambda$ has at least two parts equal to $2^{v_j(\lambda)}$.
\end{itemize}
\end{theorem}

\begin{proof} We prove the theorem by induction. 
We first show that all partitions in $\ST(n)$ satisfy conditions (a), (b), and (c). 
By inspection, this is true for $n=1,2,3$. Let $n\geq 4$ be an integer and assume that, for all $m<n$, all partitions in $\ST(m)$ satisfy (a), (b), and (c). Let $\lambda\in \ST(n)$.

If $n=2t$ for some $t>1$, then $\lambda=2\mu$ for some $\mu \in \ST(t)$. Then, $v(\lambda)=(v_1(\mu)+1, v_2(\mu)+1, \ldots, v_{\ell(\mu)}(\mu)+1)$. Since $\mu$ satisfies (a), (b), and (c), and the parts of $\lambda$ with $2$-adic valuation equal to $v_i(\mu)+1$ are precisely twice the parts of $\mu$ with $2$-adic valuation equal to $v_i(\mu)$, it follows that $\lambda$ satisfies (a), (b), and (c). 

If $n=2t+1$ for some $t>1$, then we have three cases. 
\begin{enumerate}
 \item $\lambda=\mu\sqcup (1)$ for some $\mu\in \ST(2t)$. Then, $v_{\ell(\mu)}(\mu)>0$ and $$v(\lambda)=(v_1(\mu), v_2(\mu), \ldots, v_{\ell(\mu)}(\mu),0).$$ Since $\mu$ satisfies (a), (b), and (c), so does $\lambda$. 
 
 \item $\lambda=\mu\sqcup (1,1,1)$ for some $\mu\in \ST(2t-2)$. Then, $v_{\ell(\mu)}(\mu)>0$ and $$v(\lambda)=(v_1(\mu), v_2(\mu), \ldots, v_{\ell(\mu)}(\mu),0,0,0).$$ The only parts with $2$-adic valuation zero are the three parts equal to $1$. Since $\mu$ satisfies (a), (b), and (c), so does $\lambda$. 
 
 \item $\lambda=\mu^{o+2}$ for some $\mu\in \ST(2t-1)$. Then $v(\lambda)=v(\mu)$ and from Remark \ref{R1} and the induction hypothesis, $\lambda$ satisfies (a), (b), and (c). 
\end{enumerate}

Next, we show by induction that all partitions of $n$ satisfying (a), (b), and (c) are in $\ST(n)$. By inspection, for $n=1,2,3$, the partitions in $\ST(n)$ are the only partitions satisfying conditions (a), (b), and (c). Let $n\geq 4$ be an integer and assume that if $m<n$ all partitions of $m$ satisfying (a), (b), and (c) are in $\ST(m)$. Let $\lambda$ be a partition of $n$ satisfying (a), (b), and (c).

If $n=2t$ for some $t>1$, by (b), $\lambda$ cannot have odd parts. The partition $\lambda{/2}$ has $v(\lambda{/2})=(v_1(\lambda)-1, v_2(\lambda)-1, \ldots, v_{\ell(\lambda)}(\lambda)-1)$ and thus $\lambda{/2}$ satisfies conditions (a) and (b). If $v_i(\lambda{/2})$ has multiplicity three, then $v_i(\lambda)$ has multiplicity three and $\lambda$ has at least two parts equal to $2^{v_i(\lambda)}$. Then $\lambda{/2}$ has at least two parts equal to $2^{v_i(\lambda)-1}$ and $\lambda{/2}$ satisfies (c). By induction, $\lambda{/2}\in \ST(t)$ and, by definition, $\lambda\in \ST(2t)$.

If $n=2t+1$ for some $t>1$, $\lambda$ must have odd parts, i.e, $v_{\ell(\lambda)}(\lambda)=0$. By (b), $\lambda$ has one or three odd parts and, by (c), if $\lambda$ has three odd parts, at least two odd parts must equal $1$. We transform $\lambda$ into a partition $\mu$ as follows. 
\begin{itemize}
 \item[(i)] If $m_\lambda(1)=1$, let $\mu=\lambda\setminus(1)$. Then $\mu\vdash 2t$.
 \item[(ii)] If $m_\lambda(1)=3$, let $\mu=\lambda\setminus(1,1,1)$. Then $\mu\vdash 2t-2$.
 \item[(iii)] If $m_\lambda(1)=2$, then $\mu$ is the partition obtained from $\lambda$ by subtracting two from the unique odd part greater than $1$. Then, $\mu\vdash 2t-1$.
 \end{itemize}
 It is easily seen that in each case the partition $\mu$ satisfies conditions (a), (b), and (c). By the induction hypothesis, $\mu$ is a semi-tribonacci partition. By the definition of $\mu$ from $\lambda$, it follows that $\lambda\in \ST(n)$.
\end{proof}

Theorem \ref{T-st} allows us to prove Corollary \ref{st} combinatorially. 
\begin{proof}[Combinatorial proof of Corollary \ref{st}]
If $n>0$, we define a bijection $h:\OBST(n)\to \ST(n)$ as follows. Let $\lambda\in \OBST(n)$.

If $2^k$ is non-overlined in $\lambda$, merge all parts equal to $2^k$ into a single part. Since $2^k$
 has odd multiplicity in $\lambda$ we get a part with $2$-adic valuation equal to $k$. 
 
 If $2^k$ is overlined in $\lambda$ and $2^k$ appears $2t+1$ times ($t>0$) in $\lambda$, remove the overline and merge $2t-1$ copies of $2^k$ into a single part. We get three parts with $2$-adic valuation $k$ and at least two of these parts are equal to $2^k$. (If $t=1$, all three parts are equal to $2^k$. If $t>1$, the parts with $2$-adic valuation $k$ are $2^k, 2^k$, and $2^k(2t-1)$.)
 
 By construction $h(\lambda)$ satisfies conditions (b) and (c) of Theorem \ref{T-st}. 
 Since the largest part of $\lambda$ is not overlined, in $\mu$ we get a single part with largest $2$-adic valuation and thus $h(\lambda)$ also satisfies condition (a). Therefore $h(\lambda)\in \ST(n)$.
 
For $\mu\in \ST(n)$, $h^{-1}(\mu)$ is the overpartition in $\OBST(n)$ whose parts are the parts of $\rep_2(\mu)$ with part $2^k$ overlined if and only if there are three parts in $\mu$ with $2$-adic valuation $k$. 
 \end{proof}
 The next example illustrates the bijection $h$ in the proof of Corollary \ref{st}.
\begin{example} 
 Let $\lambda=(8,8,8,8,8, \bar 4, 4,4,4,4,2,2,2,\bar 1,1,1,1,1)\in \OBST(71)$. Then $h(\lambda)=(40, 12,6, 4,4, 3,1,1)$. Conversely, if $\mu=(40, 12,6, 4,4, 3,1,1)=(2^3\cdot 5, 2^2\cdot 3 ,2\cdot 3, 2^2,2^2, 2^0\cdot3,2^0,2^0)\in \ST(71)$, then $$\rep_2(\mu)=(8,8,8,8,8,4,4,4,4,4,2,2,2,1,1,1,1,1).$$ Since in $\mu$ there are three parts with $2$-adic valuation $2$ and three parts with $2$-adic valuation $0$, we overline one part equal to $2^2$ and one part equal to $2^0$ to get $$h^{-1}(\mu)=(8,8,8,8,8, \bar 4, 4,4,4,4,2,2,2,\bar 1,1,1,1,1).$$
\end{example}

We mentioned in section \ref{prelim} that $\rep_2$ and $\bin$ are not injective as functions on $\P(n)$. However, the mapping $\rb:\ST(n)\to \B(n)\times \B(n)$ defined by $\rb (\lambda)= (\rep_2(\lambda), \bin(\lambda))$ is injective. 

 \begin{theorem} \label{rep-bin} Let $n$ be a non-negative integer. The mapping $\rb$ is one-to-one on $\ST(n)$. 
 \end{theorem}

 \begin{proof} 
We prove the theorem by induction on $n$. The statement is true for $n=0,1,2$ since in these cases $\ST(n)$ consists of binary partitions. Let $n>2$ be an integer and suppose $\rb$ is one-to-one on $\ST(m)$ for all $m<n$. Let $\lambda, \mu\in \ST(n)$ and assume $$\rb(\lambda)= \rb(\mu)=:(\alpha, \beta).$$
 By Theorem \ref{T-st}, $\alpha\in \OB(n)$; also, $\beta$ is a binary partition of $n$. Note that if $u$ is an integer with $\val_2(u)=i$, then $2^i$ is the smallest part in both $\rep_2(u)$ and $\bin(u)$. Let $2^k$ be the smallest part of both $\beta$ and $\alpha$. Thus $k$ is the smallest part in both $v(\lambda)$ and $v(\mu)$. 
 Let $b$ be the multiplicity of $2^k$ in $\beta$. 
 By Theorem \ref{T-st}, it follows that $b=1$ or $3$. Let $a$ be the multiplicity of $2^k$ in $\alpha$, so $a$ is odd. By the definition of $\bin$, if $b=1$, then $a2^k$ is the unique part with valuation $k$ in each of $\lambda$ and $\mu$. If $b=3$, then the parts with $2$-adic valuation $k$ in both $\lambda$ and $\mu$ are precisely $(a-2)2^k, 2^k, 2^k$. Let $\widetilde \lambda$ and $\widetilde \mu$ be the partitions obtained from $\lambda$ and $\mu$ by removing the parts with $2$-adic valuation $2^k$. Let $$\widetilde \alpha=\alpha \setminus (\underbrace{2^k, 2^k, \ldots, 2^k}_{a \text{ times}})$$ and $$\widetilde \beta=\begin{cases}\beta\setminus (2^k) & \text{ if } b=1\\\beta\setminus (\bin((a-2)2^k)\sqcup (2^k, 2^k))& \text{ if } b=3.\end{cases}$$ It follows from Theorem \ref{T-st} that $\widetilde\lambda, \widetilde \mu\in \ST(n-a2^k)$. Moreover, $\widetilde\alpha=\rep_2(\widetilde\lambda)=\rep_2(\widetilde\mu)$ and $\widetilde\beta= \bin(\widetilde\lambda)= \bin(\widetilde\mu)$. By the induction hypothesis, $\widetilde\lambda=\widetilde\mu$, and thus $\lambda=\mu$. 
 \end{proof}

 \begin{example} This table shows the partitions $\lambda$ in $\ST(5)$ with the corresponding distinct pairs of binary partitions. 
$$\begin{array}{c|c|c} \lambda& \rep_2(\lambda) & \bin(\lambda)\\ \hline 5 & 1+1+1+1+1 & 4+1 \\ 3+2 & 2+1+1+1 & 2+2+1 \\ 4+1 & 4+1 & 4+1 \\ 2+1+1+1 & 2+1+1+1 & 2+1+1+1 \end{array} $$
 
 \end{example}

 \begin{remark} Working as in the proof of Theorem \ref{rep-bin}, $\bin(\lambda)$ gives the necessary information for overlining parts in $\rep_2(\lambda)$. If the part $2^k$ with smallest valuation in $\bin(\lambda)$ has multiplicity $3$ in $\bin(\lambda)$, then the first instance of $2^k$ is overlined in $\rep_2(\lambda)$. Then, one works with $\widetilde \alpha$ and $\widetilde \beta$ to determine if the first instance of $2^j$, where $j$ is the smallest valuation in $\widetilde\beta$ should be overlined in $\rep_2(\lambda)$, and so on. 
 \end{remark}

\subsection{Parity Results for $\st(n)$}

We establish parity results for the sequence $\st(n)$, and therefore for the sequences $|\OBST(n)|$ and $|\ST(n)|$. Here and throughout, we write $\equiv$ to mean congruence modulo $2$.

\begin{theorem}\label{cong_dt} For $n\geq 0$, $\st(2n+1)+\st(4n+1)\equiv 0$.
\end{theorem}
\begin{proof} We prove the statement by induction. The statement is true for $n=0$ by inspection. Let $n\geq 1$ and assume that $\st(2k+1)+\st(4k+1)\equiv 0$ for all $k<n$. From the recurrence, we obtain 
\begin{align*}\st(2n+1)& =\st(2n)+\st(2n-1)+\st(2n-2)\\ & =\st(4n)+\st(4n-2)+\st(4n-4).\end{align*}Also, 
\begin{align*}\st(4n+1) & = \st(4n)+\st(4n-1)+\st(4n-2)\\ & = \st(4n)+\st(4n-2)+\st(4n-3)+\st(4n-4)+\st(4n-2).\end{align*}
Then, \begin{align*}\st(2n+1)+\st(4n+1)& \equiv \st(4n-3)+\st(4n-2)\\ & =\st(2(n-1)+1)+ \st(4(n-1)+1) \equiv 0.\end{align*}

\end{proof}

\begin{corollary} For $n\geq 0$, $\st(16n+7)\equiv \st(16n+8)\equiv \st(16n+4)$.
\end{corollary}
\begin{proof}The congruence follows from the recurrence relation and the statement of Theorem \ref{cong_dt}. We omit the details. 
\end{proof}

\subsection{Total Number of Parts in $\ST(n)$}

Let $\pst(n)$ be the total number of parts in the partitions of $\ST(n)$, whose initial conditions and recurrences imply that $\pst(1)=1$, $\pst(2)=1$, $\pst(3)=3$, and for $n>1$, 
\begin{align}  \pst(2n)\notag & =\pst(n),\\ \pst(2n+1) \notag & = \pst(2n)+\st(2n)+\pst(2n-1)+\pst(2n-2) +3\st(2n-2) \\ \label{pstnum} & = \pst(n)+\st(n)+\pst(2n-1)+\pst(n-1)+3\st(n-1). \end{align}

The following table shows the values of the sequence $\pst(n)$ for $n = 1, 2, \ldots, 12$.

$$\begin{array}
{c|ccccccccccccc}
n & 1 & 2 & 3 & 4 & 5 & 6 & 7 & 8 & 9 & 10 & 11 & 12 \\
\hline
\pst(n) & 1 & 1 & 3 & 1 & 9 & 3 & 18 & 1 & 29 & 9 & 46 & 3
\end{array} $$

\begin{theorem} For $n\geq 0$, \begin{equation}\label{pst-c}\mathrm{pst}(4n+1)\equiv \mathrm{pst}(4n+2).\end{equation}\end{theorem}

\begin{proof}

We use induction. If $n=1$, we have 
$\pst(1)=1=\pst(2)$ and \eqref{pst-c} holds. Assume \eqref{pst-c} is true with $n$ replaced by $n-1$, that is, $$\pst(4n-3)\equiv
\pst(4n-2).$$
The left-hand side of \eqref{pst-c} equals 
\begin{align*}
\pst(4n+1)& \equiv\pst(2n)+\st(2n)+\pst(4n-1)+\pst(2n-1)+\st(2n-1)\\ & = \pst(n)+\st(n)+\pst(2(2n-1)+1)+\pst(2n-1)+\st(2n-1)\\ & =\pst(n)+\st(n)+\pst(2n-1)+\st(2n-1)+\pst(2(2n-1)-1)\\ & \qquad \qquad +\pst((2n-1)-1)+3\st((2n-1)-1)+\pst(2n-1)+\st(2n-1)\\ & =\pst(n)+\st(n)+2\pst(2n-1)+4\st(2n-1)+\pst(4n-3)\\ & \qquad \qquad +\pst(2n-2)+\st(2n-2))\\ & =\pst(n)+\st(n)+2\pst(2n-1)+4\st(2n-1)+\pst(4n-3)\\ & \qquad \qquad +\pst(n-1)+\st(n-1))\\ & \equiv
\pst(n)+\st(n)+\pst(4n-3)+\pst(n-1)+\st(n-1).
\end{align*}
The right-hand side of \eqref{pst-c} equals 
\begin{align*} \pst(4n+2)& =\text{pst}(2n+1)\\ & =\pst(n)+\st(n)+\pst(2n-1)+\pst(n-1)+3\st(n-1).
\end{align*}
Adding the two sides,
\begin{align*}
\pst(4n+1)+\pst(4n+2)& =\pst(n)+\st(n)+\pst(4n-3)+\pst(n-1)+\st(n-1)\\ & \qquad +\pst(n)+\st(n)+\pst(2n-1)+\pst(n-1)+3\st(n-1)\\ & \equiv
\pst(2n-1)+\pst(4n-3)\equiv 0.\end{align*}
\end{proof}
\section{Semi-Padovan partitions} \label{sec_spa}

\begin{definition}  The semi-Padovan sequence, $\{\spa(n)\}_{n\geq 0}$ is defined recursively by $\spa(0)=1$, $\spa(1)=0$, and for $n\geq 1$, \begin{align*}\spa(2n)&=\spa(n),\\ \spa(2n+1)&=\spa(2n-1)+\spa(2n-2).\end{align*} 
\end{definition}

The following table shows the values of the semi-Padovan sequence $\spa(n)$ for $n = 0, 1, \ldots, 12$.
$$\begin{array}
{c|ccccccccccccc}
n & 0 & 1 & 2 & 3 & 4 & 5 & 6 & 7 & 8 & 9 & 10 & 11 & 12 \\
\hline
\spa(n) & 1 & 0 & 0 & 1 & 0 & 1 & 1 & 1 & 0 & 2 & 1 & 2 & 1
\end{array} $$
This is the sequence $\{f(n)\}_{n\geq 0}$ in Section \ref{sec_gen} with $a_0=1$, $a_1=0$, and $c_1=0$, $c_2=c_3=1$, Thus, we are in Case 2 of Section \ref{sec_gen} and \eqref{small} becomes $$F(x) = \prod_{r=0}^{\infty} \left(1+\frac{x^{3\cdot 2^r}}{1-x^{2^{r+1}}}\right),$$
which is the generating function for $|\OB_{\mathrm{R}}(n)|$, where $\OB_{\mathrm{R}}(n)$ is the set of odd binary partitions of $n$ such that each part is repeated, that is, each part occurs with multiplicity at least $3$. For example, $$\OB_{\mathrm{R}}(9)=\{2 + 2 + 2 + 1 + 1 + 1, \, 1 + 1 + 1 + 1 + 1 + 1 + 1 + 1 + 1 \}.$$

\begin{definition}
The semi-Padovan partitions are defined recursively by $\SPa(0)=\{(\, )\}$ (i.e., the set containing the empty partition),
$\SPa(1)=\emptyset$, $\SPa(3)=\{(3)\}$, and \begin{align*}\SPa(2n)& = 2 \SPa(n), \text{ \ \ if } n\geq 1, \\ \SPa(2n+1) & = \SPa(2n-1)^{o+2} \bigcup \SPa(2n-2) \sqcup (1,1,1), \text{ \ \ if } n\geq 2.\end{align*}\end{definition}

 \begin{example} The table below shows the partitions in $\SPa(n)$ for $n=1,\ldots,16$.  Note that $\SPa(n)$ is empty if $n$ is a power of $2$.
$$\begin{array}{c|c} n& \SPa(n) \\ \hline 1 & \\ 2 & \\ 3 & 3 \\ 4& \\ 5 & 5 \\ 6&6 \\ 7 & 7 \\ 8 & \\ 9 & 9, 6+1+1+1 \\ 10 & 10 \\ 11 & 11, 6+3+1+1 \\ 12 & 12 \\ 13 & 13, 6+5+1+1, 10+1+1+1 \\ 14 & 14 \\ 15 & 15, 7+6+1+1, 10+3+1+1, 10+3+1+1, 12+1+1+1 \\ 16 & \end{array} $$
 \end{example}
 
If $n$ is odd, partitions in $\SPa(n)$ have at least one odd part. If $n$ is even, partitions in $\SPa(n)$ have only even parts. Thus, partitions in $\SPa(2n-1)^{o+2}$ have an odd part greater than $1$ while partitions in $\SPa(2n)\sqcup (1,1,1)$ have no odd part greater than $1$. %Since $$\SPa(2n-1)^{o+2}\, \bigcap \ \SPa(2n)\sqcup (1,1,1)=\emptyset.$$ 
Therefore the union in the definition of $\SPa(2n+1)$ is disjoint, so $|\SPa(n)|=\spa(n)$. Then, the statement of the next theorem follows from the interpretation of the generating function for $\spa(n)$.

\begin{theorem}\label{SP-OB3} For $n\geq 1$, $|\SPa(n)|=|\OB_\R(n)|$.\end{theorem}

The next theorem introduces a non-recursive description of the partitions in $\SPa(n)$. It will lead to a combinatorial proof of Theorem \ref{SP-OB3}. We use the notation introduced before Theorem \ref{T-st}.

\begin{theorem}\label{spa-nonrec} Let $n>0$. The set $\SPa(n)$ consists of all partitions $\lambda$ of $n$ satisfying the following three conditions:
\begin{itemize}\item[(a)] $v_1(\lambda)$ has multiplicity one in $v(\lambda)$; 
\item[(b)] if $1<j\leq \ell(\lambda)$, then $v_j(\lambda)$ has multiplicity three and 
$\lambda$ has at least two parts equal to $2^{v_j(\lambda)}$;
\item[(c)] $2^{v_1(\lambda)}$ is not a part of $\lambda$.
\end{itemize}
\end{theorem}
\begin{proof} The theorem is proved by induction. The argument is similar to that of Theorem \ref{T-st} and we omit the details.
\end{proof}

\begin{corollary} \label{C_spa} Let $\SF'(n)$ be the subset of semi-Fibonacci partitions with no part equal to a power of $2$, or equivalently, the set of partitions of $n$ into parts with different $2$-adic valuations and such that no part is a power of $2$. Then $|\SPa(n)|=|\SF'(n)|$ for all $n\geq 1$.
\end{corollary}

\begin{proof}
It is easy to see that $\rep_2$ is a bijection from $\SF'(n)$ to $\OB_{\mathrm{R}}(n)$; for the inverse, merge equal parts into a single part.
\end{proof}

\begin{proof}[Combinatorial proof of Theorem \ref{SP-OB3}]
Corollary \ref{C_spa} leads to a combinatorial proof of Theorem \ref{SP-OB3}. We define a bijection $\varphi:\SF'(n)\to \SPa(n)$ as follows. If $\lambda \in \SF'(n)$, replace each part $2^k m$, $m>1$ odd, of $\lambda$ with $k\neq v_1(\lambda)$ by parts $2^k(m-2), 2^k, 2^k$. We obtain a partition in $\SPa(n)$. The inverse, $\varphi^{-1}$, takes a partition $\mu\in \SPa(n)$ and merges all equal parts into a single part to obtain a partition in $\SF'(n)$.

Then, by Corollary \ref{C_spa}, the map $\rep_2\circ \varphi^{-1}:\SPa(n)\to\OB_{\mathrm{R}}(n)$ is a bijection. 
\end{proof}
Next, we define 
a modified semi-Padovan sequence. 
\begin{definition}The modified semi-Padovan sequence,$\{\spa'(n)\}_{n\geq 0}$, is defined recursively by $\spa'(0)=0$, $\spa'(1)=1$, and for $n\geq 1$, \begin{align*}\spa'(2n)&=\spa'(n),\\ \spa'(2n+1)&=\spa'(2n-1)+\spa'(2n-2).\end{align*} \end{definition}
Note that the only difference between the definitions of $\spa(n)$ and $\spa'(n)$ is that the initial conditions are switched: $\spa(0)=1$, $\spa(1)=0$, while $\spa'(0)=0$, $\spa'(1)=1$.

The following table shows the values of the modified semi-Padovan sequence $\spa'(n)$ for $n = 0, 1, \ldots, 12$.
$$\begin{array}
{c|ccccccccccccc}
n & 0 & 1 & 2 & 3 & 4 & 5 & 6 & 7 & 8 & 9 & 10 & 11 & 12 \\
\hline
\spa'(n) & 0 & 1 & 1 & 1 & 1 & 2 & 1 & 3 & 1 & 4 & 2 & 5 & 1
\end{array} $$

This is the sequence $\{f(n)\}_{n\geq 0}$ of Section \ref{sec_gen} with $a_0=0$, $a_1=1$, and $c_1=0$, and $c_2=c_3=1$. Thus, we are in Case 1 of Section \ref{sec_gen} and \eqref{big} becomes 
$$F(x) =\sum_{i=0}^\infty \frac{x^{2^i}}{1-x^{2^{i+1}}} \prod_{r=0}^{i-1} \left(1+\frac{x^{3\cdot 2^r}}{1-x^{2^{r+1}}}\right).$$

Then, $F(x)$ is the generating function for $|\OB_{\mathrm{R}'}(n)|$, where $\OB_{\mathrm{R}'}(n)$ is the set of odd binary partitions of $n$ such that only the largest part can have multiplicity one. For example, \begin{align*}\OB_{\mathrm{R}'}(9)=\{& 4+1+1+1+1+1,\,  2+1+1+1+1+1+1+1, \\ &  2+2+2+1+1+1, \, 1+1+1+1+1+1+1+1+1\}.\end{align*}

\begin{definition}
The modified semi-Padovan partitions are defined recursively by \\ $\SPa'(0)=\emptyset$, $\SPa'(1)=\{(1)\}$, and for $n\geq 1$, \begin{align*}\SPa'(2n)& = 2 \SPa'(n),\\ \SPa'(2n+1) & = \SPa'(2n-1)^{o+2} \bigcup \SPa'(2n-2) \sqcup (1,1,1).\end{align*}\end{definition}
Again the only differences in the definitions of $\SPa(n)$ and $\SPa'(n)$ are in the initial conditions.

 \begin{example} The table below shows the partitions in $\SPa'(n)$ for $n=1,..,12$.
$$\begin{array}{c|c} n& \SPa'(n) \\ \hline 1 & 1 \\ 2 & 2\\ 3 & 3 \\ 4& 4 \\ 5 & 5,2+1+1+1 \\ 6&6 \\ 7 & 7, 3+2+1+1,4+1+1+1 \\ 8& 8 \\9 & 9, 4+3+1+1,5+2+1+1,6+1+1+1 \\ 10& 10, 4+2+2+2 \\ 11& 11, 5+4+1+1, 6+3+1+1, 7+2+1+1, 8+1+1+1 \\ 12& 12 \end{array} $$
 \end{example}

The same argument as in the case of $\SPa(n)$ shows that $|\SPa'(n)|=\spa'(n)$ and proves the next theorem.
\begin{theorem}\label{SPA'-OB'3} For $n\geq 1$, $|\SPa'(n)|=|\OB_{\mathrm{R}'}(n)|$.
\end{theorem}
Next, we give a non-recursive description of the partition in $\SPa'(n)$. This will allow us to prove Theorem \ref{SPA'-OB'3} combinatorially.

\begin{theorem}\label{spa'-nonrec} Let $n>0$. The set $\SPa'(n)$ consists of all partitions $\lambda$ of $n$ satisfying the following two conditions: 
\begin{itemize}\item[(a)] $v_1(\lambda)$ has multiplicity one in $v(\lambda)$; 
\item[(b)] if $1<j\leq \ell(\lambda)$, then $v_j(\lambda)$ has multiplicity three and 
$\lambda$ has at least two parts equal to $2^{v_j(\lambda)}$.
\end{itemize}
\end{theorem}
\begin{proof}The theorem is proved by induction. The argument is similar to that of Theorem \ref{T-st} and we omit the details. 
\end{proof}

\begin{proof}[Combinatorial proof of Theorem \ref{SPA'-OB'3}] We define a bijection $\psi:\SPa'(n)\to \OB_{\mathrm{R}'}(n)$ as follows. If $\lambda \in \SPa'(n)$, we define $\psi(\lambda)=\rep_2(\lambda)\in \OB_{\mathrm{R}'}(n)$. The inverse, $\psi{-1}$, takes a partition $\mu\in \OB_{\mathrm{R}'}(n)$ and merges all occurrences of the largest part into a single part, and for each $2^k<\mu_1$ it replaces all $m_\mu(2^k)$ parts equal to $2^k$ by the three parts: $2^k(m_\mu(2^k)-2), 2^k, 2^k$. Then $\psi^{-1}(\mu)\in \SPa'(n)$.
\end{proof}

\begin{remark} The conditions (a) and (b) in Theorem \ref{spa'-nonrec} are the same as in Theorem \ref{spa-nonrec}. The partitions in $\SPa(n)$ are precisely the partitions in $\SPa'(n)$ that do not have the part with the largest $2$-adic valuation equal to a power of $2$. 
Moreover, Theorems \ref{T-st} and \ref{spa'-nonrec} show that $\SPa'(n)\subseteq \ST(n)$ for every $n\geq 1$.
Thus, for $n\geq 1$, we have $$\SPa(n)\subseteq \SPa'(n)\subseteq \ST(n).$$
 It is straightforward to see that for $n\geq 1$ we have $\OB_{\mathrm{R}}(n)\subseteq \OB_{\mathrm{R}'}(n)
\subseteq \OBST(n)$. \end{remark}

\subsection{Parity Results for $\spa(n)$ and $\spa'(n)$}

We begin by showing that the parity of $\spa(n)$ is completely determined by residues modulo $7$.
This result is similar to the result of \cite[Theorem 2]{A19} for the semi-Fibonacci sequence, namely: $\sf(n)$ is even if and only if $n\equiv 0 \pmod 3$. 

We say that two finite sequences are congruent modulo $2$ if and only if the corresponding terms are congruent.

\begin{theorem}\label{cong_spa} 
For $n\geq 0$, $\spa(n)\equiv 0$ if and only if $n\equiv 1, 2, 4 \pmod 7$.
\end{theorem}

\begin{proof} We say that two finite sequences are congruent modulo $2$ if and only if the corresponding terms are congruent. We use complete induction to prove that for all $n\geq 0$, $\{\spa(7n+k)\}_{k=0,1,\ldots,6} \equiv (1,0,0,1,0,1,1).$
We work with $14$ consecutive values at a time. For the base cases, the first seven values of $\spa(n)$, starting with $n=0$, are $y:=(1,0,0,1,0,1,1)$ and the next seven values are $(1,0,2,1,2,1,3)\equiv y$, as expected. 

Let $M>0$ be an integer. Assume that, for all $m<M$ and $k=0,1,\ldots,13$, $\spa(14m+k)$ is even if and only if $k=1,2,4,8,9,11$. Next, we calculate $S:=\{\spa(14M+k)\}_{k=0,1,\ldots,13}$.

For the seven odd arguments, we use the recurrences and induction one value at a time to find the parities. The first value is \begin{align*}\spa(14M+1)& =\spa(14M-1)+\spa(14M-2)\\ & =\spa(14(M-1)+13)+\spa(14(M-1)+12) \\ & \equiv 1+1\equiv 0.\end{align*} The other six are similar; we get $$\{\spa(14M+2j-1)\}_{j=1,2,\ldots,7} \equiv (0,1,1,1,0,0,1) =: x.$$ 

The seven even arguments are easier; we use induction once for all the values:
$$\{\spa(14M+2j)\}_{j=0,1,\ldots,6}=\{\spa(7M+j)\}_{j=0,1,\ldots,6}\equiv y.$$
Interleaving $x$ and $y$ gives
 \begin{align*}S\equiv(1,0,0,1,0,1,1,1,0,0,1,0,1,1) \equiv(y,y),\end{align*} as required. 
\end{proof}
Theorem \ref{cong_spa} leads to the next two corollaries. 

\begin{corollary}\label{cong_spa_2}
For all $n\geq 0$, 
\begin{align}\label{spacong3} \spa(2n+1)+\spa(4n+1)+\spa(8n+1)& \equiv 0,\\ \label{spacong1}\spa(4n+1)+\spa(4n+2)+\spa(4n+3)+\spa(4n+6)& \equiv 0. \end{align}

\end{corollary}

\begin{proof} We prove \eqref{spacong3} and omit the  proof of \eqref{spacong1} which uses recurrences and Theorem \ref{cong_spa}. 

Let $n\geq 0$ and write $n=7m+r$ with $0\leq r\leq 6$. Then, \begin{align*}\spa(2n+1)& + \spa(4n+1)+\spa(8n+1)\\ & =\spa(14m+2r+1)+\spa(28m+4r+1)+\spa(7(8m+r)+r+1). \end{align*} The table below shows the values of $2r+1$, $4r+1$ and $r+1$ modulo $7$ for the corresponding values of $r$. 
 
$$\begin{array}{c|c|c|c} r & 2r+1\!\!\!\! \pmod 7 & 4r+1\!\!\!\! \pmod 7 & r+1\!\!\!\! \pmod 7\\ \hline 0 & 1& 1& 1\\ 1& 3& 5& 2\\ 2& 5& 2& 3\\ 3 & 0 & 6 & 4 \\ 4& 2& 3& 5 \\ 5 & 4& 0& 6\\ 6 & 6 & 4& 0
\end{array}$$ Using Theorem \ref{cong_spa} in each case completes the proof of \eqref{spacong3}. 
\end{proof}

\begin{corollary}\label{cong_span} For $n\geq 0$, \begin{equation*} \spa(8n+7)\equiv \spa(n).\end{equation*}
\end{corollary}

\begin{proof}Since $8n+7\equiv n\pmod 7$, the statement follows from Theorem \ref{cong_spa}. 
\end{proof}

Although we have not found a companion to Theorem \ref{cong_spa} for $\spa'(n)$, there is one for the first congruence of Corollary \ref{cong_spa_2}. 

\begin{proposition}\label{cong_spa'_2} For $n\geq 0$, $\spa'(2n+1)+\spa'(4n+1)+\spa'(8n+1)\equiv 1$. 
\end{proposition}

We also have a companion to Corollary \ref{cong_span} for $\spa'(n)$.

\begin{theorem} \label{cong_spa'n} For $n\geq 0$, \begin{equation}\label{spa'-cong}\spa'(8n+7)\equiv \spa'(n)+1.\end{equation}
\end{theorem}
\begin{proof}
Using the recurrences and Proposition \ref{cong_spa'_2},we have \begin{align*}\spa'(8n+7) & = 
\spa'(8n+5)+\spa'(8n+4) \\ & =
\spa'(8n+3)+\spa'(8n+2)+\spa'(8n+4) \\ & = 
\spa'(8n+1)+\spa'(8n)+\spa'(4n+1)+\spa'(2n+1) \\ & \equiv
\spa'(n)+1. \end{align*} 
\end{proof}

\begin{proposition} For $n\geq 0$, \begin{align*}\spa'(4n+1)+\spa'(4n+4) & \equiv \spa'(4n+5)+\spa'(4n+10) \\ & \equiv
\spa'(4n+6)+\spa'(4n+7) \\ & \equiv
\spa'(4n+8)+\spa'(4n+11) \\ & \equiv
\spa'(4n+12)+\spa'(4n+17) \\ & \equiv
\spa'(4n+13)+\spa'(4n+14) \\ & \equiv
\spa'(4n+15)+\spa'(4n+18) \\ & \equiv
\spa'(4n+9).
\end{align*}
\end{proposition}

\begin{proposition} For $n\geq 0$, $\spa'(2n)+\spa'(2n+1) + \spa'(2n+3)\equiv 0$.  
\end{proposition}

\subsection{Total Number of Parts in $\SPa(n)$ and $\SPa'(n)$}

Let $\pspa(n)$ be the total number of parts in the partitions of $\SPa(n)$, whose initial conditions and recurrences imply that $\pspa(1)=0$,  $\pspa(2)=0$,  $\pspa(3)=1$,  and for $n>1$,
\begin{align} \notag
 \pspa(2n)& \notag =\pspa(n),\\ 
 \pspa(2n+1)& \label{pspanum} =\pspa(2n-1)+\pspa(2n-2)+3\spa(2n-2).\end{align}

The total number of parts in the partitions of $\SPa'(n)$, $\pspa'(n)$, is defined exactly like $\pspa(n)$ except for the initial values. We have   $\pspa'(1) =1$, $\pspa'(2)=1$, $\pspa'(3)=1$ and  for $n>1$, 
\begin{align}\notag
\pspa'(2n)& =\pspa'(n),\\ \label{pspa'num} \pspa'(2n+1)& =\pspa'(2n-1)+\pspa'(2n-2)+3\spa'(2n-2).\end{align}

The following table shows the values of the sequences $\pspa(n)$ and $\pspa'(n)$ for $n = 1, 2, \ldots, 12$.

$$\begin{array}
{c|ccccccccccccc}
n & 1 & 2 & 3 & 4 & 5 & 6 & 7 & 8 & 9 & 10 & 11 & 12 \\
\hline
\pspa(n) & 0 & 0 & 1 & 0 & 1 & 1 & 1 & 0 & 5 & 1 & 5 & 1 \\
\pspa'(n) & 1 & 1 & 1 & 1 & 5 & 1 & 9 & 1 & 13 & 5 & 17 & 1
\end{array} $$

\begin{proposition}
For all $n\geq 0$, 
\begin{align} \label{pspacong1}\pspa(4n+1)+\pspa(4n+2)+\pspa(4n+3)+\pspa(4n+6)& \equiv 0, \\ \label{pspa'cong1}\pspa'(4n+1)+\pspa'(4n+2)+\pspa'(4n+3)+\pspa'(4n+6)& \equiv 0. \end{align}

\end{proposition}

\begin{theorem} \label{cong_pspa} For $n\geq 1$, 
\begin{align}\label{pspa1}
\pspa(8n+2)& \equiv \pspa(8n-5),\\ \label{pspa'1}
 \pspa'(8n+2)& \equiv \pspa'(8n-5).\end{align}
\end{theorem}
\begin{proof} 
From \eqref{pspanum} and \eqref{pspa'num}, it follows that for $n>1$ we have \begin{align}\label{pspaequiv} \pspa(2n+1) & \equiv\pspa(2n-1)+\pspa(2n-2)+\spa(2n-2), \\ \pspa'(2n+1) & \label{pspa'equiv}\equiv\pspa'(2n-1)+\pspa'(2n-2)+\spa'(2n-2) .\end{align}

Now we can prove \eqref{pspa1} by induction. From the table above we see that $$\pspa(8\cdot 1+2)\equiv \pspa(8 \cdot 1-5).$$ Let $n\geq 2$ be an integer and assume that $$\pspa(8(n-1)+2)\equiv \pspa(8(n-1)-5).$$ Using \eqref{pspaequiv} four times, we have \begin{align*}
\pspa(8n-5) \equiv\, & \pspa(8n-13)\\ & + \pspa(8n-8)+ \pspa(8n-10)+ \pspa(8n-12)+ \pspa(8n-14)\\ & + \spa(8n-8)+ \spa(8n-10)+ \spa(8n-12)+ \spa(8n-14)\\ = \, & \pspa(8n-13)\\ & + \pspa(4n-4)+ \pspa(4n-5)+ \pspa(4n-6)+ \pspa(4n-7)\\ & + \spa(4n-4)+ \spa(4n-5)+ \spa(4n-6)+ \spa(4n-7).\end{align*}Similarly, \begin{align*}
\pspa(8n+2) & = \pspa(4n+1)\\ & \equiv \pspa(4n-3) +\pspa(4n-2)+\pspa(4n-4) \\ & \qquad \qquad\qquad \quad +\spa(4n-2)+\spa(4n-4)\\ & = \pspa(8n-6) +\pspa(4n-2)+\pspa(4n-4) \\ & \qquad \qquad\qquad \quad +\spa(4n-2)+\spa(4n-4).
\end{align*} Using the induction hypothesis together with \eqref{spacong1} and \eqref{pspacong1} completes the proof of \eqref{pspa1}. 
The proof of \eqref{pspa'1} is similar. 
\end{proof}

Recall that in Theorem \ref{cong_spa} we showed that the parity of $\spa(n)$ is completely determined by the residue of $n$ modulo $7$. In Corollary \ref{cong_spa_2}, we showed that $\spa(2n+1)+\spa(4n+1)+\spa(8n+1)$ is always even. The next theorem establishes the parity of $\pspa(2n+1)+\pspa(4n+1)+\pspa(8n+1)$ in terms of the residue of $n$ modulo $7$.

\begin{theorem} For $n\geq 0$, $\pspa(2n+1)+\pspa(4n+1)+\pspa(8n+1)\equiv 0$  if and only if $n\equiv 0,4,5 \pmod 7$.
% $$\{\pspa(2n+1)+\pspa(4n+1)+\pspa(8n+1)\}_{n \equiv 0,1,\ldots,6 \!\!\! \! \pmod 7} \equiv (0,1,1,1,0,0,1).$$ 
\end{theorem}
%\begin{proof}

% \gcom{$\spa(0)=1$ by Definition 3, so we cannot set it to $0$ by convention. I have a fix, as follows. Note I extend the base case to include $n=7$. This makes the arguments in the sum in \eqref{pspaequiv} all positive.}

% \gcom{First note that applying \eqref{pspaequiv} $k$ times (as long as $k<n$) gives \begin{equation}\label{pspak-rel} \pspa(2n+1)\equiv \pspa(2n-2k+1)+\sum_{j=1}^{k}(\pspa(2n-2j)+\spa(2n-2j)).\end{equation} 

% We use induction to prove the theorem. For $0\leq n\leq 7$, the statement can be verified directly. Let $n\geq 8$. 
% Using \eqref{pspak-rel} $7$ times, we obtain...
% BUT the $-13, -27, -55$ should be $-13, -13, -13$. On the positive side, I have a proof: }

% \

\begin{proof}
We use induction. Let $$r(n)=\pspa(2n+1)+\pspa(4n+1)+\pspa(8n+1).$$ The base cases are $n=0,...,6$, for which $r(n)$ has respective values $0,1,1,1,0,0,1$, as expected. Let $n\geq 6$ and assume $r(n)$ has the correct parity. Then, using the recurrences several times,

\begin{align*} r(n + 1) & 
 =
\pspa(2 n + 3) + \pspa(4 n + 5) + \pspa(8 n + 9)\\ &
 \equiv
\pspa(2 n + 1) + \pspa(n) + \spa(n) \\ & \quad +
\pspa(4 n + 3) + \pspa(2 n + 1) + \spa(2 n + 1) \\ & \quad +
\pspa(8 n + 7) + \pspa(4 n + 3) + \spa(4 n + 3)\\ & 
 \equiv
\pspa(2 n + 1) + \pspa(n) + \spa(n) \\ &  \quad +
 \pspa(4 n + 1) + \pspa( n) + \spa( n) + \pspa(2 n + 1) + 
 \spa(2 n + 1)  \\ & \quad +
 \pspa(8 n + 5) + \pspa(2 n + 1) + \spa(2 n + 1) + \pspa(4 n + 3) + 
 \spa(4 n + 3)\\ & \equiv   
 \pspa(4 n + 1) + 
 \pspa(8 n + 5) + \pspa(2 n + 1) +   \pspa(4 n + 3) + 
 \spa(4 n + 3).\end{align*}

 Continuing with the recurrences in a similar manner, we obtain 
$$r(n+1)\equiv \pspa(2 n + 1) + \pspa(4 n + 1) + \pspa(8 n + 1) + \spa(n).$$
 
 That is, $r(n+1) \equiv r(n)+\spa(n)$. 
 
We use the induction hypothesis and Theorem 5.6 to see that $r(n+1)$ has the required parities. In the table below, the values of $r(n)$, $\spa(n)$, and $r(n+1)$ are modulo $2$.

$$\begin{array}
{ccccc}
n+1 & n & r(n) & \spa(n) & r(n+1) \\
\hline
0 & 6 & 1 & 1 & 0 \\
1 & 0 & 0 & 1 & 1 \\
2 & 1 & 1 & 0 & 1 \\
3 & 2 & 1 & 0 & 1 \\
4 & 3 & 0 & 0 & 0 \\
5 & 4 & 0 & 0 & 0 \\
6 & 5 & 0 & 1 & 1
\end{array} $$

\end{proof}

\begin{theorem} For $n\geq 0$, $$\pspa'(2n+1) +\pspa'(4n+1)+\pspa'(8n+1)\equiv \spa'(2n+1).$$\end{theorem} 
\begin{proof} Using \eqref{pspa'equiv} we rewrite $\pspa'(8n+1)$ in terms of $\pspa'(8n-5)$. Then, using \eqref{pspa1}, the recurrences for $\pspa'(n)$ and $\spa'(n)$, and simplifying, the statement follows.
\end{proof}
Next, we present an interesting congruence modulo $3$ that connects the four sequences of this section. 

\begin{theorem} For $n\geq1$, $$\pspa'(n)-\pspa(n)\equiv \spa'(n)-\spa(n) \pmod 3.$$
\end{theorem}
\begin{proof}
One can check that the statement is true for $n=1,2,3$. Let $n\geq 3$ and assume that $\pspa'(k)-\pspa(k)\equiv \spa'(k)-\spa(k) \pmod 3$
for all $k<n$.
If $n=2t$, then 
\begin{align*} \pspa'(2t)-\pspa(2t) & = \pspa'(t)-\pspa(t) \\ & \equiv \spa'(t)-\spa(t) \pmod 3 \\ & =\spa'(2t)-\spa(2t) \\ & =\spa'(n)-\spa(n) .
 \end{align*}

 If $n=2t+1$, then 
\begin{align*} \pspa'(2t+1)& -\pspa(2t+1) \\ & \equiv \pspa'(2t-1)+ \pspa'(2t-2)\\ & \qquad \qquad -(\pspa(2t-1) +\pspa(2t-2))\pmod 3\\ & \equiv \spa'(2t-1)- \spa(2t-1)+\spa'(2t-2) -\spa(2t-2)\pmod 3\\ & = \spa'(2t+1)-\spa(2t+1).
 \end{align*}

\end{proof}

\section{Semi-Pell partitions} \label{sec_sp} 
\begin{definition}The semi-Pell sequence $\{\sp(n)\}_{n\geq 0}$ is defined recursively by $\sp(0)=0, \sp(1)=1$, and for $n\geq 1$, \begin{align*}\sp(2n)&=\sp(n),\\ \sp(2n+1)&=2\sp(2n)+\sp(2n-1).\end{align*} \end{definition}

The following table shows the values of the semi-Pell sequence $\sp(n)$ for $n = 0, 1, \ldots, 12$.

$$\begin{array}
{c|ccccccccccccc}
n & 0 & 1 & 2 & 3 & 4 & 5 & 6 & 7 & 8 & 9 & 10 & 11 & 12 \\
\hline
\sp(n) & 0 & 1 & 1 & 3 & 1 & 5 & 3 & 11 & 1 & 13 & 5 & 23 & 3
\end{array} $$

It is easily seen by induction that $\sp(n)$ is odd for all $n\geq 1$.

This is the sequence $\{f(n)\}_{n\geq 0}$ in Section \ref{sec_gen} with $a_0=0$, $a_1=1$, $c_1=2$, $c_2=1$, and $c_3=0$. Thus, we are in Case 1 of Section \ref{sec_gen} and \eqref{big} becomes 

\begin{align*}F(x) & =\sum_{i=0}^\infty \frac{x^{2^i}}{1-x^{2^{i+1}}} \prod_{r=0}^{i-1} \left(1+\frac{2x^{2^r}}{1-x^{2^{r+1}}}\right)\\
& =\sum_{i=0}^\infty \frac{x^{2^i}}{1-x^{2^{i+1}}} \prod_{r=0}^{i-1} \left(1+\frac{x^{2^r}}{1-x^{2^{r+1}}}+ \frac{x^{2^r}}{1-x^{2^{r+1}}}\right)\end{align*}

Let $\overline\OB'(n)$ be the set of odd binary overpartitions of $n$ in which the largest part must be overlined. For example, $$\overline\OB'(5)=\{\overline{4}+1,\, \overline{4}+\overline{1},\, \overline{2}+1+1+1,\, \overline{2}+\overline{1}+1+1,\, \overline{1}+1+1+1+1\}.$$ By an argument similar to that of Section \ref{sec_st}, $F(x)$ is the generating function for the sequence $| \overline\OB'(n)|$.

A composition of $n$ is an ordered list of positive integers whose sum is $n$. In \cite{KM}, the authors defined the semi-Pell compositions, which are enumerated by the semi-Pell sequence, as follows. “Let $\OC(n)$ be the set of weakly unimodal binary compositions of $n$ such that each part size occurs together, or “in one place,” an odd number of times. In other words every part size lies in a distinct ‘colony’. For example, members of $\OC(45)$ include, using the frequency notation, $(16, 4^3, 2^3, 1^{11}), (2^5, 4, 8^3, 1^7), (1^7, 8, 16, 2^7)$, but the following weakly unimodal binary compositions of $45$ do not belong to $\OC(45)$: $(2^3,4^3,16,2,1^9),(1^3,2^5,4,8^3,1^4)$." In the definition above, 'colony' refers to the set of parts either all before or all after of the maximum part of the unimodal composition. 

Our definition of $\sp(n)$ differs from that in \cite{KM} only for $n=0$; in \cite{KM}, $\sp(0)=1$. If $n>0$, there is an obvious correspondence between $\OC(n)$ and $\overline\OB'(n)$. If $\lambda \in \overline\OB'(n)$, and $2^k$ is an overlined part of $\lambda$ less than $\lambda_1$, remove the overline and place all parts equal to $2^k$ before all copies of the largest part. If $2^k$ appears in $\lambda$ but it is not overlined, place all parts equal to $2^k$ after all copies of the largest part. The parts before the largest parts are arranged in non-decreasing order and the parts after the largest parts are arranged in non-increasing order.

\begin{definition}
The semi-Pell partitions $\SP(n)$ are defined recursively by $\SP(0)=\emptyset$, 
 $\SP(1)=\{(1)\}$, and for $n\geq 1$, \begin{align*}\SP(2n)& =2\SP(2n),\\ 
\SP(2n+1)& =\add_{1,1}(\SP(2n)) \bigcup \SP(2n)\sqcup(1) \bigcup \SP(2n-1)\sqcup(1,1)\end{align*}
\end{definition}

 \begin{example} The table below shows the partitions in $\SP(n)$ for $n=1,\ldots,8$ written in compact form.
$$\begin{array}{c|c} n& \SP(n) \\ \hline 1 & 1 \\ 2 & 2\\ 3 & 3, 21, 111 \\ 4& 4 \\ 5 & 5,41,311,2111,11111 \\ 6&6,42,222 \\ 7 & 7,52,322,61,421,2221,511,4111,31111,211111,1111111 \\ 8& 8 \end{array} $$
 \end{example}
 
The semi-Pell numbers $\sp(n)$ enumerate $\SP(n)$. 
This is clear for even $n$; for odd $n\geq 3$, the three subsets in the definition of $\SP(n)$ are pairwise disjoint because their partitions contain no $1$s, one $1$, or more than one $1$, respectively. Hence, we have the following theorem. 

\begin{theorem}\label{SP-OB'} For $n\geq 0$, $$|\SP(n)|=|\overline\OB'(n)|.$$\end{theorem}
The next theorem enables us to provide a combinatorial proof of Theorem \ref{SP-OB'}. 
\begin{theorem} \label{SP-bin}For $n \geq 0$, $\lambda \in \SP(n)$ if and only if $\bin(\lambda_1) \sqcup\lambda\setminus(\lambda_1)$ is an odd binary partition.
\end{theorem}

\begin{proof}
This can be proved by induction, using $\bin(2k)=2\bin(k)$, $\bin(2k+1)=2\bin(k)\sqcup (1)$, and the fact that all parts in partitions in $\SP(2n)$ are even.
\end{proof}
\begin{proof}[Combinatorial proof of Theorem \ref{SP-OB'}] Let $n\geq 0$. We define a bijection $\xi:\SP(n)\to \overline\OB'(n)$ as follows. If $\lambda\in SP(n)$, define $\xi(\lambda)\in \overline\OB'(n)$ to be the overpartition whose parts are the parts of $\bin(\lambda_1) \sqcup(\lambda\setminus(\lambda_1))=\bin(\lambda_1) \sqcup (\lambda_2, \lambda_3, \ldots)$ and the overlined parts are precisely the parts of $\bin(\lambda_1)$. The largest part in $\bin(\lambda_1)$ is at least $\lambda_2$ and therefore $\xi(\lambda)\in \overline\OB'(n)$. 

For the inverse, start with $\mu\in \overline \OB'(n)$. To obtain $\xi^{-1}(\mu)$, replace all overlined parts in $\mu$ by a part equal to their sum. That part is greater than or equal to any non-overlined part of $\mu$. By Theorem \ref{SP-bin}, the resulting partition is in $\SP(n).$
\end{proof}

\subsection{A Congruence Connecting $\sp(n)$ to the Paper-folding Sequence}\label{paper_folding}

Next we consider another sequence related to $\sp(n)$. 
Let $z(n)$ be sequence \href{https://oeis.org/A014707}{A014707} in \cite{OEIS} (the paper-folding sequence). It is defined for $n\geq 0$ by $z(4n)=0, z(4n+2)=1$, and $z(2n+1)=z(n)$. Note that $z(n)\in \{0,1\}$ for all $n\geq 0$. 
We have the following parity result.

\begin{theorem} If $n\geq 1$, we have $$\frac{\sp(n)-1}{2}\equiv z(n-1). $$
\end{theorem}
\begin{proof}We prove the theorem by induction. Let $b(n)=\displaystyle \frac{\sp(n)-1}{2}$. Since $b(1)=b(2)=0$ and $z(0)=z(1)=0$, the theorem holds for $n=1,2$. Let $n$ be a positive integer and assume the $b(m)\equiv z(m-1)$ for all $m<n$.

If $n=2t$ is even, using the recurrence for the semi-Pell sequence and the induction hypothesis, we have \begin{align*}b(n)=b(2t)& =b(t) \equiv z(t-1). \end{align*} On the other hand, using the recurrence for $z(n)$, we have $z(t-1)=z(2(t-1)+1)=z(2t-1)=z(n-1)$. Hence, $b(n)\equiv z(n-1)$.

If $n=2t+1$ is odd, we have \begin{align*} b(2t+1)& =\frac{\sp(2t+1)-1}{2}\\ & = \frac{2\sp(2t)+\sp(2t-1)-1}{2}\\ & = \frac{2\sp(2t)-2+\sp(2t-1)-1+2}{2}\\ & = 2b(2t)+b(2t-1)+1\\ & \equiv z(2t-2)+1\\ & \equiv z(2t).\end{align*} For the last line we used the fact that $z(2t-2)$ and $z(2t)$ have opposite parity. 

\end{proof}

\subsection{Total Number of Parts in $\SP(n)$}

For $n \geq 1$, let $\psp(n)$ be the total number of parts in $\SP(n)$. 
From the initial condition and recurrences for $\SP(n)$, $\psp(1) =1$, $\psp(2)=1$,  and for $n\geq 1$, 
\begin{align} \notag \psp(2n)& =\psp(n),\\ \label{pspnum} \psp(2n+1)& =\sp(2n)+2\psp(2n)+2\sp(2n-1)+\psp(2n-1).\end{align}

The following table shows the values of the sequence $\psp(n)$ for $n = 1, \ldots, 12$.

$$\begin{array}
{c|cccccccccccc}
n & 1 & 2 & 3 & 4 & 5 & 6 & 7 & 8 & 9 & 10 & 11 & 12 \\
\hline
\psp(n) & 1 & 1 & 6 & 1 & 15 & 6 & 40 & 1 & 65 & 15 & 126 & 6
\end{array} $$

\begin{proposition}
For $n\geq0$, $\psp(2n+1)\equiv 1$ if and only if $n\equiv 0$.
\end{proposition}

Recall from subsection \ref{paper_folding} that $z(n)$ is the paper-folding sequence.

\begin{theorem}
For $n \geq 1$, $\psp(n) \equiv z(n-1) + 1$.
\end{theorem}

\begin{proof}
Since $\sp(2n)$ is odd, we have $$\psp(2n+1)\equiv \psp(2n-1)+1.$$ We prove the theorem by induction. 
Since $z(0)=z(1)=1$ and $\psp(1)=\psp(2)=0$, the statement holds for $n=1,2$. 
Assume $\psp(k) \equiv z(k-1) + 1$ for all $k<n$. 
If $n$ is even, i.e., $n=2t$ for some $t\in \mathbb N$, we have \begin{align*}\psp(n)= \psp(2t)=\psp(t)& \equiv z(t-1)+1 \\ & = z(2t-1)+1 \\ & =z(n-1)+1. \end{align*} 
If $n$ is odd, i.e., $n=2t+1$ for some $t\in \mathbb N$, we have \begin{align*}\psp(n)= \psp(2t+1)\equiv \psp(2t-1)+1 & \equiv z(2t-2) \\ & \equiv z(2t)+1 \\ & = z(n-1)+1. \end{align*} Above, we used again the fact that $z(2t-2)$ and $z(2t)$ have opposite parity. 
\end{proof}

\section{Semi-Narayana cow partitions}\label{sec_snc}

\begin{definition}The semi-Narayana cows sequence $\{\snc(n)\}_{n\geq 0}$ is defined recursively  by 
$\snc(0)=\snc(1)=1$, and for $n\geq 1$, \begin{align*}\snc(2n)&=\snc(n),\\ \snc(2n+1)&=\snc(2n)+\snc(2n-2).\end{align*} \end{definition}

 The following table shows the values of the semi-Narayana cows sequence $\snc(n)$ for $n = 0, 1, \ldots, 12$.
$$\begin{array}
{c|ccccccccccccc}
n & 0 & 1 & 2 & 3 & 4 & 5 & 6 & 7 & 8 & 9 & 10 & 11 & 12 \\
\hline
\snc(n) & 1 & 1 & 1 & 2 & 1 & 2 & 2 & 3 & 1 & 3 & 2 & 3 & 2
\end{array} $$

This is the sequence \href{https://oeis.org/A120562}{A120562} in \cite{OEIS}. The original sequence is Narayana's cows sequence \href{https://oeis.org/A000930}{A000930} in \cite{OEIS}.

This is the sequence $\{f(n)\}_{n\geq 0}$ in Section \ref{sec_gen} with 
 $a_0=a_1=1$, $c_1=1$, $c_2=0$, and $c_3=1$, Thus, we are in Case 2. and \eqref{small} becomes $$F(x)=\prod_{r=0}^\infty (1+x^{2^r}+x^{3\cdot 2^r}).$$

This is the generating function for $|\OB_{1,3}(n)|$, where $\OB_{1,3}(n)$ is the set of odd binary partitions of $n$ in which parts appear one or three times. For example, $\OB_{1,3}(7)=\{4+2+1,\, 4+1+1+1,\, 2+2+2+1\}$

\begin{definition}Semi-Narayana cow partitions are defined recursively as $\SC(0)=\{(\ )\}$, $\SC(1)=\{(1)\}$, and for $n\geq 1$, \begin{align*}\SC(2n)& =2\SC(n),\\ \SC(2n+1)& =\SC(2n)\sqcup(1) \bigcup \SC(2n-2)\sqcup(3).\end{align*} \end{definition} Thus $|\SC(n)|=\sc(n)$ for all $n\geq 0$. 

 \begin{example} The table below shows the partitions in $\SNc(n)$ for $n=1,\ldots,12$.
$$\begin{array}{c|c} n& \SNc(n) \\ \hline 1 & 1 \\ 2 & 2\\ 3 & 2+1,3 \\ 4& 4 \\ 5 & 4+1,3+2 \\ 6&4+2,6 \\ 7 & 4+2+1,6+1,4+3 \\ 8& 8 \\9& 8+1,4+3+2,6+3\\10& 8+2,6+4\\11& 8+2+1,6+4+1,8+3 \\12& 8+4,12 \end{array} $$
 \end{example}

 Next, we give a non-recursive description of semi-Narayana cow partitions. 
 \begin{theorem}
     For $n\geq 0$, $\SC(n)$ is the set of partitions of $n$ whose parts are of the form $2^k$ or $3\cdot 2^k$ for some $k\geq 0$ and have distinct $2$-adic valuations. 
 \end{theorem}
 \begin{proof}
     The theorem is proved by induction. We omit the details. 
 \end{proof}

\begin{corollary} For $n\geq 0$, $\SC(n)\subseteq \SF(n)$.
    \end{corollary} 

\begin{remark}For $n\geq 0$, the function $\rep_2$ is a bijection between $\SC(n)$ and the set of odd binary partitions of $n$ in which parts appear once or three times.\end{remark}

\subsection{Parity Result for $\snc(n)$}

\begin{theorem} \label{snc-parity}
For $n \geq 1$, $\snc(n) \equiv 0$ if and only if $n \equiv 3, 5, 6 \pmod 7$.
\end{theorem}

\begin{proof}
We omit the proof by induction, which is similar to that of Theorem \ref{cong_spa}.
\end{proof}

\subsection{Total Number of Parts in $\SNc(n)$}

Let $\psnc(n)$ be the total number of parts in $\SNc(n)$.

From the initial conditions and recurrences for $\SNc(n)$,
the sequence $\psnc(n)$ can be defined directly by $\psnc(0)=0$, $\psnc(1)=1$ and, for $n\geq 1$, 
\begin{align*}
\psnc(2n)& =\psnc(n),\\ \psnc(2n+1)&=\psnc(2n)+\psnc(2n-2)+\snc(2n+1).
\end{align*}

The following table shows the values of the sequence $\psnc(n)$ for $n = 0,1, \ldots, 12$.

$$\begin{array}
{c|cccccccccccccc}
n & 0 & 1 & 2 & 3 & 4 & 5 & 6 & 7 & 8 & 9 & 10 & 11 & 12 \\
\hline
\psnc(n) & 0 & 1 & 1 & 3 & 1 & 4 & 3 & 7 & 1 & 7 & 4 & 8 & 3
\end{array} $$

\begin{proposition} 
For $n\geq0$, $$\psnc(2n+1) + \psnc(2n+5) + \psnc(4n+1) + \psnc(4n+9) + \psnc(8n+1) \equiv 0.$$
\end{proposition}
\begin{proof} 
The proof uses the recurrence for $\psnc(2n+1)$ and Theorem \ref{snc-parity}.
\end{proof}

\section{Delayed Semi-Fibonacci partitions} \label{sec_dsf}

Let $\dsf(n)$ be the delayed semi-Fibonacci sequence, which uses the same recurrences as the semi-Fibonacci sequence but with different initial conditions. 
\begin{definition}The delayed semi-Fibonacci sequence, $\{\dsf(n)\}_{n\geq 0}$, is defined recursively by $\dsf(0)=\dsf(1)=\dsf(2)=0, \dsf(3)=1$ and for $n\geq 2$, \begin{align*}\dsf(2n)&=\dsf(n),\\ \dsf(2n+1)&=\dsf(2n)+\dsf(2n-1).\end{align*}\end{definition}

The following table shows the values of the delayed semi-Fibonacci sequence $\dsf(n)$ for $n = 0, 1, \ldots, 12$.

$$\begin{array}
{c|ccccccccccccc}
n & 0 & 1 & 2 & 3 & 4 & 5 & 6 & 7 & 8 & 9 & 10 & 11 & 12 \\
\hline
\dsf(n) & 0 & 0 & 0 & 1 & 0 & 1 & 1 & 2 & 0 & 2 & 1 & 3 & 1
\end{array} $$

Let $\OB_1(n)$ be the set of odd binary partitions of $n$ such that the largest part must be repeated. 
For example, $\OB_1(7)=\{2 + 2 + 2 + 1, \, 1 + 1 + 1 + 1 + 1 + 1 + 1\}.$

\begin{theorem} \label{dsf} For $n\geq 0$, $|\OB_1(n)|=\dsf(n)$. 
\end{theorem}

\begin{proof}[Combinatorial proof] The condition that the first part must be repeated implies that $|\OB_1(n)|=0$ for $n=0,1,2$. Since $\OB_1(3)=\{(1,1,1)\}$, we have $|\OB_1(3)|=1$. Next we show that for $n\geq 4$, the sequence $\OB_1(n)$ satisfies the same recurrences as $\dsf(n)$. 

If $n$ is even and $\lambda\in \OB_1(n)$, then $1$ is not a part of $\lambda$ (since the multiplicity of $1$ in $\lambda$ is either zero or odd), so $ \lambda/2$ is a partition. As the property that the largest part of a partition is repeated is  preserved by doubling or halving all parts, $ \lambda/2\in \OB_1(n/2)$ if and only if $\lambda\in \OB_1(n)$, so $|\OB_1(2n)|= |\OB_1(n)|$.

 If $n=2t+1$, $t\geq 2$, and $\lambda \in \OB_1(n)$, then $\lambda$ has at least one part equal to $1$. If  $m_\lambda(1)=1$, then $1$ is not the largest partof $\lambda$ and $\lambda\setminus (1)\in \OB_1(2t)$. If  $m_\lambda(1)\geq 3$, Since $n\geq 4$, either $1$ is the largest part and $m_\lambda(1)\geq 3$, or $1$ is not the largest part. 
 Then $\lambda\setminus (1,1)\in\OB_1(2t-1)$. This transformation is invertible:  if $\mu\in \OB_1(2t)$, then $\lambda:=\mu\sqcup (1) \in \OB_1(2t+1)$ with $m_{\lambda}(1)=1$; and if $\mu\in \OB_1(2t-1)$, then $\lambda:=\mu\sqcup(1,1) \in \OB_1(2t+1)$ with $m_{\lambda}(1)\geq 3$. 
 Thus, for $t\geq 2$, we have $|\OB_1(2t+1)|=|\OB_1(2t)|+|\OB_1(2t-1)|$. 
 
 Therefore, $|\OB_1(n)|=\dsf(n)$ for $n\geq 0$.
\end{proof}

\begin{proof}[Analytic proof] 

Since the recurrence is delayed, we cannot use the generating function of section \ref{sec_gen}. However, if $$G(x):=\sum_{n=0 }^\infty\dsf(n)x^n,$$ working as in section \ref{sec_gen}, we obtain 
\begin{align*}G(x)& = \sum_{i= 0}^\infty \frac{x^{3\cdot 2^i}}{1-x^{2^{i+1}}} \prod_{r=0}^{i-1}\left(1+ \frac{x^{2^r}}{1-x^{2^{r+1}}}\right)\\ & = \sum_{n= 0}^\infty|\mathcal{OB}_1(n)|x^n.
\end{align*}

To see the last equality, arguing informally,  for each $i\geq 0$, $\displaystyle \frac{x^{3\cdot 2^i}}{1-x^{2^{i+1}}}$ generates parts equal to $2^i$ with odd multiplicity at least three, while $\displaystyle \prod_{r=0}^{i-1}\left(1+ \frac{x^{2^r}}{1-x^{2^{r+1}}}\right)$ generates  parts equal to $2^{r}$, $r<i$, each with odd multiplicity. 
\end{proof}

\begin{definition}
The delayed semi-Fibonacci partitions are defined recursively by $\DSF(1)=\DSF(2)=\emptyset, \DSF(3)=\{(2,1)\},$ and for $n\geq 2$, \begin{align*}\DSF(2n)& =2 \DSF(n),\\ \DSF(2n+1) & = (\DSF(2n)\sqcup (1)) \cup \DSF(2n-1)^{o+2}.\end{align*}\end{definition}

 \begin{example} The table below shows the partitions in $\DSF(n)$ for $n=1,\ldots,12$. 
$$\begin{array}{c|c} n& \DSF(n) \\ \hline  1 & \\ 2 & \\ 3 & 2+1 \\ 4&  \\ 5 & 3+2 \\ 6&4+2 \\ 7 & 5+2, 4+2+1 \\ 8 & \\ 9 & 7+2,4+3+2 \\ 10 & 6+4 \\ 11 & 9+2, 5+4+2, 6+4+1 \\ 12 & 8+4 \end{array} $$
\end{example}

Let $\OB_2(n)$ be the set of odd binary partition of $n$ with the first two parts being consecutive powers of $2$. For example, $\OB_2(7)=\{4 + 2 + 1, \, 2 + 1 + 1 + 1 + 1 + 1\}.$ The function $\xi:\OB_2(n)\to\OB_1(n)$ that splits the first  part of $\lambda\in \OB_2(n)$ into two equal parts is a bijection.

\begin{theorem} \label{dsf-ob1} For $n\geq 1$, the function $\xi\circ\rep_2$ is a bijection from $\DSF(n)$ to $\OB_1(n)$. 
\end{theorem}
\begin{proof} We prove  by induction that $\rep_2$ is a bijection from $\DSF(n)$ to $\OB_2(n)$. Since $\OB_2(1)=\OB_2(2)=\emptyset$, the statement holds for $n=1,2$. By definition $\DSF(3)=\OB_2(3)=\{(2,1)\}$ and $\rep_2(2,1)=(2,1)$. Let $n\geq 4$ and assume that $\rep_2$ is a bijection from $\DSF(k)$ to $\OB_2(k)$ for all $k<n$.

Start with $\lambda\in \DSF(n)$. If $n$ is even, $\rep_2(\lambda)=2\rep_2(\lambda{/2})$. Since $\lambda{/2}\in\DSF(n/2)$, we have $\rep_2(\lambda{/2})\in \OB_2(n/2)$. Multiplying each part of $\rep_2(\lambda{/2})$ by $2$ preserves the odd multiplicity of the parts and the fact that the first two parts are consecutive powers of $2$. Thus $\rep_2(\lambda)\in \OB_2(n)$. If $n$ is odd, then $\lambda$ there are two cases. 

(i) $\lambda=\mu\sqcup(1)$ for some $\mu \in \DSF(2n)$. By definition, there is no odd part in $\mu$. Then $\rep_2(\mu)\in \OB_2(n)$ and there is no part equal to $1$ in $\rep_2(\mu)$. Then, $\rep_2(\lambda)=\rep_2(\mu)\sqcup (1)\in \OB_2(n)$. 

(ii) $\lambda=\mu^{o+2}$ for some $\mu \in \DSF(2n-1)$. Then, the odd part of $\lambda$ is at least $3$ and $\rep_2(\lambda)=\rep_2(\mu)\sqcup(1,1)$. Since $\rep_2(\mu) \in \OB_2(2n-1)$, it follows that $\rep_2(\lambda)\in \OB_2(n)$. 

The inverse of $\rep_2$ is the map $g$ that merges equal parts into a single part. We show inductively that $g$ maps partitions in $\OB_2(n)$ to partitions in $\DSF(n)$. Suppose $g:\OB_2(k)\to \DSF(k)$ for all $k<n$ and let $\eta\in \OB_2(n)$. If $n$ is even, $\eta$ has no part equal to $1$. Then $\eta/2\in \OB_2(n/2)$ and $g(\eta/2)\in \DSF(n/2)$. Thus, $2g(\eta/2)=g(\eta)\in \DSF(n)$. If $n$ is odd, $\eta$ has an odd number of parts equal to $1$ and we have two cases. 

(a) $\eta$ has exactly one part equal to $1$. Since $n\geq 4$, we have $\eta\setminus (1)\in \OB_2(n-1)$ and $g(\eta\setminus (1))\in \DSF(n-1)$ has no odd part. Then $g(\eta)=g(\eta\setminus (1))\sqcup (1) \in \DSF(n)$. Note that the odd part of $g(\eta)$ is equal to $1$. 

(b) $\eta$ has at least three parts equal to $1$. Then $\eta\setminus (1,1)\in \OB_2(n-2)$ and $g(\eta\setminus (1,1))\in \DSF(n-2)$. Then $g(\eta)=g(\eta\setminus(1,1))^{o+2} \in \DSF(n)$. Note that the odd part of $g(\eta)$ is at least $3$. 
\end{proof}

Theorem \ref{dsf-ob1} leads to the following non-recursive definition of delayed semi-Fibonacci partitions. 

\begin{corollary} Let $n>0$. The set $\DSF(n)$ consists of partitions $\lambda$ of $n$ into parts with distinct $2$-adic valuations such that, if $v_1(\lambda)$ is the largest $2$-adic valuation of parts in $\lambda$, then $\lambda$ has part $2^{v_1(\lambda)}$ and also a part with $2$-adic valuation $v_1(\lambda)-1$. 
\end{corollary}
\begin{remark} Clearly $\DSF(n)\subseteq \SF(n)$. 
\end{remark}
\subsection{Parity Results for $\dsf(n)$}

\begin{theorem} \label{dsf-equiv}For $n\geq 1$, $\dsf(4n-3)\equiv \dsf(4n+3)$.
\end{theorem}
\begin{proof} The statement is true for $n=1$ by inspection. For $n\geq 2$, the recurrence relation implies \begin{align*}\dsf(4n+3) & =\dsf(4n+2)+\dsf(4n)+\dsf(4n-2)+\dsf(4n-3)\\& = \dsf(2n+1)+\dsf(2n)+\dsf(2n-1)+\dsf(4n-3)\\ &= \dsf(2n+1)+\dsf(2n+1)+\dsf(4n-3) \\ & \equiv \dsf(4n-3).\end{align*}
\end{proof}

\begin{proposition} For $n\geq 1$, $\dsf(2n+1)+\dsf(8n+1)\equiv 1$. 
\end{proposition}

\section{Semi-Lucas partitions} 
\label{sec_sl} 

\begin{definition} The semi-Lucas sequence, $\{\sl(n)\}_{n\geq0}$, is defined recursively by $\sl(1)=2, \sl(2)=1$, and for $n\geq 1$ \begin{align*}\sl(2n)& =\sl(n)\\ \sl(2n+1)& =\sl(2n)+\sl(2n-1).\end{align*} \end{definition}

The following table shows the values of the semi-Lucas sequence $\sl(n)$ for $n = 1, \ldots, 12$.

$$\begin{array}
{c|cccccccccccc}
n & 1 & 2 & 3 & 4 & 5 & 6 & 7 & 8 & 9 & 10 & 11 & 12 \\
\hline
\sl(n) & 2 & 1 & 3 & 1 & 4 & 3 & 7 & 1 & 8 & 4 & 12 & 3
\end{array} $$

\begin{theorem}\label{eqn1_sl} For $n\geq 2$, $\sl(n)=\sf(n)+\dsf(n)$. 
\end{theorem}
\begin{proof} The theorem follows easily by induction. We omit the proof. 
\end{proof}

The generating function for $\sl(n)$ is obtained from the generating functions for $\sf(n)$ and $\dsf(n)$ with correction terms for the coefficients of $x^0$ and $x^1$: 

\begin{align*}\sum_{n=1}^\infty \sl(n)x^n & = x-1+ \prod_{r=0}^\infty \left(1+\frac{x^{2^r}}{1-x^{2^{r+1}}}\right) \\ & \hspace{1.3cm} +\sum_{i= 0}^\infty \frac{x^{3\cdot 2^i}}{1-x^{2^{i+1}}}\prod_{r=0}^{i-1} \left(1+\frac{x^{2^r}}{1-x^{2^{r+1}}}\right)
 \\ & = x+\sum_{n= 1}^\infty|\OB(n)|x^n+\sum_{n= 1}^\infty|\OB_1(n)|x^n\\ &=\sum_{n=1}^\infty |\overline\OB^*(n)|x^n, 
\end{align*} 
where $\overline\OB^*(1)=\{(1), (\overline 1)\}$ and, for $n\geq 2$, the set $\overline\OB^*(n)$ consists of odd binary overpartitions of $n$ in which only the largest part may be overlined, and, moreover, the largest part may be overlined only if  it is repeated. For example, $$\overline\OB^*(5)=\{4+1,\, 2+1+1+1,1+1+1+1+1,\, \overline{1}+1+1+1+1\}.$$

\begin{definition} The semi-Lucas partitions are defined recursively by $ \SL(2)=\{(2)\}$, $\SL(3)=\{(3), (2,1), (1,1,1)\},$ and for $n> 1$, \begin{align*} \SL(2n)& = 2\SL(n)\\ \SL(2n+1) & = \SL(2n)\sqcup (1) \bigcup \SL(2n-1)\sqcup (1,1).\end{align*}\end{definition}

Clearly, for $n\geq 2$ we have $|\SL(n)|=\sl(n)$. 

 \begin{example} The table below shows the partitions in $\SL(n)$ for $n=2,\ldots,8$. 
$$\begin{array}{c|c} n& \SL(n) \\ \hline  2 & 2\\ 3 & 3, 21,111 \\ 4& 4 \\ 5 & 41,311,2111,11111 \\ 6&42,222 \\ 7 & 61,421,2221,4111,31111,211111,1111111 \\ 8 & 8 \end{array} $$
\end{example}

\begin{theorem} Let $n>1$. The set $\SL(n)$ consists of all partitions $\lambda$ of $n$ satisfying the following conditions: \begin{itemize} \item[(a)] $v_i(\lambda)$ has odd multiplicity for all $1\leq i\leq \ell(\lambda)$;
\item[(b)] $\lambda_i\in \{2^k \mid k \in \mathbb Z_{\geq 0}\}$ for all $1<i\leq \ell(\lambda)$;
\item[(c)] $\lambda_1\in \{2^k, 3\cdot 2^k \mid k \in \mathbb Z_{\geq 0}\}$;
\item[(d)] $\lambda_1$ has the largest $2$-adic valuation among the parts of $\lambda$.
\end{itemize}
\end{theorem}

\begin{proof} The theorem is proved by induction. We omit the details.
\end{proof}

For $n>1$, let $$\SL_{3 \nmid}(n):=\{\lambda \in \SL(n) \Big| 3 \nmid \lambda_1\} \text{ \ and \ } \SL_{3 \mid}(n):=\{\lambda \in \SL(n) \Big| \ 3 \mid \lambda_1\}.$$ Then, $\SL(n)=\SL_{3 \nmid}(n)\sqcup \SL_{3 \mid}(n)$. Clearly, if $n>1$, $\SL_{3 \nmid}(n)=\OB(n)$ and thus $|\SL_{3 \nmid}(n)|=\sf(n)$. Moreover, $\rep_2:\SF(n)\to SL_{3 \nmid}(n)$ is a bijection. 

We define a bijection $f: \OB_2(n)\to \SL_{3 \mid}(n)$ as follows. Let $\mu\in \OB_2(n)$. Then $\mu_1=2\mu_2$. Define $f(\mu)$ to be the partition obtained from $\mu$ by merging the first two parts of $\mu$, i.e., removing  parts  $\mu_1$ and $\mu_2$ and inserting a part equal to $3\mu_2$   The inverse of $f$ takes a partition $\lambda\in SL_{3 \mid}(n)$ and replaces the first part $\lambda_1=3\cdot 2^k$ by  parts $2^{k+1}$ and $2^k$. Thus, $f^{-1}(\lambda)\in \OB_2(n)$. Then, the proof of Theorem \ref{dsf-ob1} implies that $f  \circ \rep_2 :\DSF(n)\to SL_{3 \mid}(n)$ is a bijection. 
\begin{remark}
The discussion above gives a combinatorial proof of Theorem \ref{eqn1_sl}. \end{remark}
\medskip

\subsection{Parity Results for $\sl(n)$}
The proof of the next proposition is analogous to the proof of Theorem \ref{dsf-equiv}.

\begin{proposition} For $n\geq 2$, $\sl(4n-3)\equiv \sl(4n+3)$.
\end{proposition}

\begin{proposition}\label{cong2_sl}
For $n\geq 1$, $\sl(4n+1) + \sl(4n+3)\equiv \sl(4n+4) + \sl(4n+6)$.
\end{proposition}

\subsection{Total Number of Parts in $\SF(n)$, $\DSF(n)$, and $\SL(n)$}

Let $\psf(n)$ be the total number of parts in $\SF(n)$.

From the initial conditions and recurrences for $\SF(n)$,
the sequence $\psf(n)$ can be defined directly by $\psf(1)=1$ and, for $n\geq 1$, \begin{align*}\psf(2n)& =\psf(n)\\ \psf(2n+1)&=\psf(2n)+\sf(2n)+\psf(2n-1).\end{align*} 

The following table shows the values of the sequence $\psf(n)$ for $n = 1, 2, \ldots, 12$.

$$\begin{array}
{c|ccccccccccccc}
n & 1 & 2 & 3 & 4 & 5 & 6 & 7 & 8 & 9 & 10 & 11 & 12 \\
\hline
\psf(n) & 1 & 1 & 3 & 1 & 5 & 3 & 10 & 1 & 12 & 5 & 20 & 3 
\end{array} $$

\begin{theorem}For $n\geq 0$, $$\psf(8n+2)\equiv \psf(8n+5).$$
\end{theorem}
\begin{proof}
We verify directly that $\psf(2)\equiv \psf(5)$. Let $n\geq 0$ and assume $\psf(8m+2)+\psf(8m+5)\equiv 0$ for all $m<n$. Then,

\begin{align*}\psf(8n+2)& +\psf(8n+5) = \psf(8n+2)+ \psf(8n+4)+\sf(8n+4)+\psf(8n+3)\\ & = \psf(8n+2)+ \psf(8n+4)+\sf(8n+4)+\psf(8n+2)\\ & \hspace{4.5cm}+ \sf(8n+2)+\psf(8n+1)\\ & \equiv \psf(8n+4)+\sf(8n+4) + \sf(8n+2)+\psf(8n+1)\\ 
& = \psf(4n+2)+ \sf(4n+2) + \sf(4n+1)+ \psf(8n)+\sf(8n)\\ & \qquad\qquad\qquad +\psf(8n-1)\\ & = \psf(4n+2)+ \sf(4n+2) + \sf(4n)+\sf(4n-1)+ \psf(4n)\\ & \qquad\qquad\qquad +\sf(4n)+\psf(8n-2)+ \sf(8n-2)+ \psf(8n-3)\\ & \equiv \psf(4n+2)+ \sf(4n+2) +\psf(4n) + \psf(4n-1)\\ & \qquad\qquad\qquad + \psf(8n-3)\\ & = \psf(2n+1)+ \sf(2n+1) +\psf(2n) + \psf(4n-2)+\sf(4n-2)\\ & \qquad \qquad\qquad +\psf(4n-3)+ \psf(8n-3)\\ & = \psf(2n)+\sf(2n)+\psf(2n-1)+\sf(2n)+\sf(2n-1)+\psf(2n)\\ & \qquad\qquad\qquad + \psf(4n-2)+\sf(4n-2)+\psf(4n-3)+ \psf(8n-3)\\ & \equiv \psf(2n-1)+\sf(2n-1)+ \psf(4n-2)+\sf(4n-2)+\psf(4n-3)\\ & \qquad\qquad\qquad + \psf(8n-3) \\ & = \psf(4n-2)+\sf(4n-2)+ \psf(4n-2)+\sf(4n-2)+\psf(4n-3)\\ & \qquad\qquad\qquad + \psf(8n-3)\\ & \equiv \psf(4n-3)+ \psf(8n-3) \\ & = \psf(8n-6)+ \psf(8n-3)\\ & = \psf(8(n-1)+2)+ \psf(8(n-1)+5) \equiv 0.
 \end{align*}
\end{proof}

Let $\pdsf(n)$ and $\psl(n)$ be the total number of parts in $\DSF(n)$ and $\SL(n)$, respectively.

From the initial conditions and recurrences for the two sets of partitions, for $n>2$,
\begin{align} 
\pdsf(1)& =0, \pdsf(2)=0, \pdsf(3)=2, \notag \\ \pdsf(2n)& =\pdsf(n),\notag \\ \pdsf(2n+1)& =\dsf(2n)+\pdsf(2n)+\pdsf(2n-1),\label{pdsfnum} \\ \ \notag \\
\psl(2)&=1, \psl(3) =6, \notag\\ \psl(2n)& =\psl(n), \notag\\ \psl(2n+1)& =\sl(2n)+\psl(2n)+\psl(2n-1)+2\sl(2n-1). \label{pslnum}
\end{align}

The following table shows the values of the sequences $\psf(n)$, $\pdsf(n)$, and $\psl(n)$ for $n = 1, 2, \ldots, 12$.

$$\begin{array}
{c|ccccccccccccc}
n & 1 & 2 & 3 & 4 & 5 & 6 & 7 & 8 & 9 & 10 & 11 & 12 \\
\hline
\pdsf(n) & 0 & 0 & 2 & 0 & 2 & 2 & 5 & 0 & 5 & 2 & 8 & 2 \\
\psl(n) & & 1 & 6 & 1 & 14 & 6 & 31 & 1 & 47 & 14 & 81 & 6
\end{array} $$

\smallskip

\begin{proposition} We have
\begin{align*}\psf(2n+1) + \psf(4n+1) + \psf(8n+1) + \psf(16n+1) & \equiv 0 \  \text{ for  } n\geq 0,
\\ \pdsf(2n+1) + \pdsf(4n+1) + \pdsf(8n+1) + \pdsf(16n+1) & \equiv 1 \  \text{ for  } n\geq 1,\\ \psl(2n+1) + \psl(4n+1) + \psl(8n+1) + \psl(16n+1)& \equiv 1\  \text{ for  } n\geq 2.\end{align*} 
\end{proposition} \smallskip

\begin{proposition}
We have  
\begin{align*}
\pdsf(16n+3) &\equiv \pdsf(16n+12) \  \text{ for  } n\geq 0,\\
\pdsf(16n+4) &\equiv \pdsf(16n+13) \  \text{ for  } n\geq 1.\\
\end{align*}
\end{proposition}

\section{Stern-Brocot partitions} \label{sec_sb}

\begin{definition} The Stern-Brocot sequence, $\{\sb(n)\}_{n\geq 0}$, also known as the Stern diatomic sequence, is defined recusrively by $\sb(0)=0, \sb(1)=1$, and for $n\geq 1$, \begin{align*}\sb(2n)&=\sb(n),\\ \sb(2n+1)&=\sb(n+1)+\sb(n).\end{align*} \end{definition}

The following table shows the values of the Stern-Brocot sequence $\sb(n)$ for $n = 0, 1, \ldots, 12$.

$$\begin{array}
{c|ccccccccccccc}
n & 0 & 1 & 2 & 3 & 4 & 5 & 6 & 7 & 8 & 9 & 10 & 11 & 12 \\
\hline
\sb(n) & 0 & 1 & 1 & 2 & 1 & 3 & 2 & 3 & 1 & 4 & 3 & 5 & 2
\end{array} $$

In \cite{R}, Reznick proved that, for $n\geq 0$, the sequence $\sb(n+1)$ enumerates $\HB(n)$, the hyperbinary partitions of $n$, which are the binary 
 partitions of $n$ with parts repeated at most twice. Thus, as first shown in \cite{C}, the generating function for $\sb(n)$ is given by
$$x\prod_{i=0}^\infty(1+x^{2^i}+x^{2\cdot 2^i}).$$

\begin{definition} The Stern-Brocot partitions are defined recursively by $\SB(1)=\{(1)\}$ and for $n\geq 1$ \begin{align*}\SB(2n)& = 2\SB(n),\\ \SB(2n+1)& = 2\SB(n)\sqcup(1)\bigcup (2\SB(n+1))^-.
\end{align*}\end{definition}
Here $(2\SB(n+1))^- := \{(\lambda_1, \lambda_2, \ldots, \lambda_\ell)\, |\, (\lambda_1+1, \lambda_2, \ldots, \lambda_\ell)\in 2\SB(n+1)\}$.

The sequence $\SB(n)$ is well defined because partitions in $\SB(2n)$ have all parts even and partitions in $\SB(2n+1)$ have exactly one odd part, either the largest part or $1$, the smallest part. It is easily proved by induction that, for any $n\geq 1$, all partitions in $\SB(n)$ have distinct parts. 
The next result is immediate from Reznick's result.

\begin{theorem}\label{thm_sb} If $n\geq 1$, $|\HB(n-1)|=|\SB(n)|.$ 
\end{theorem}

 \begin{example} The table below shows the partitions in $\HB(n-1)$ and $\SB(n)$ for $n=1,\ldots,10$. 
$$\begin{array}{c|c|c} n & \HB(n-1) & \SB(n) \\ \hline  1 & (\ ) & 1\\2 & 1 & 2\\ 3 & 2,1+1& 3, 2+1 \\ 4&2+1& 4 \\ 5&4,2+2,2+1+1 & 5,3+2,4+1 \\ 6&4+1,2+2+1&6,4+2 \\ 7&4+2,4+1+1,2+2+1+1 & 7,6+1,4+2+1 \\ 8&4+2+1 & 8 \\9&8,4+4,4+2+2,4+2+1+1& 9,5+4,7+2,8+1 \\ 10& 8+1,4+4+1,4+2+2+1 & 10,6+4,8+2 \end{array} $$
\end{example}

 The following properties for partitions in $\SB(n)$ are easily proved by induction. 

\begin{itemize}
\item[(a)] If $\lambda=(\lambda_1, \lambda_2, \ldots, \lambda_\ell)\in \SB(n)$, then $\lambda\setminus (\lambda_1)=( \lambda_2, \ldots, \lambda_\ell)$ is a binary partition. 
(Doubling parts or inserting the part equal to $1$ in a binary partition results in a binary partition.) 

\item[(b)] For $n\geq 1$, partitions in $\SB(n)$ have distinct $2$-adic valuations of parts. Hence, $\SB(n)\subseteq \SF(n)$. 

\item[(c)] For $n\geq 1$, the partition $(n)$ with a single part is in $\SB(n)$. 
\end{itemize}

\medskip

Our next goal is to give a bijective proof of Theorem \ref{thm_sb}. 
To this end, we first establish two useful properties of the function $\bin$. 

\begin{lemma} \label{L_bin} For any positive integer $m$, we have $$\bin(2m-1)=2\bin(m-1)\sqcup (1).$$
\end{lemma}

\begin{proof} Suppose $m>0$ and $$\bin(m)=2^{k_1}+2^{k_2}+\cdots +2^{k_{j-1}}+ 2^{k_j},$$ with $k_1>k_2> \ldots >k_j\geq 0.$ Then, \begin{align*}\bin(2m-1)& =2^{k_1+1}+2^{k_2+1}+\cdots +2^{k_{j-1}+1}+\bin(2^{k_j+1}-1)\\ & = 2^{k_1+1}+2^{k_2+1}+\cdots +2^{k_{j-1}+1}+2^{k_j}+2^{k_j-1}+ \cdots +2+1. \end{align*} Similarly, $$\bin(m-1)= 2^{k_1}+2^{k_2}+\cdots +2^{k_{j-1}}+ 2^{k_j-1}+ 2^{k_j-2}+ \cdots +2+1.$$ The result follows. 
\end{proof}

\begin{corollary}\label{C_bin} For any positive even integer $m$ we have $$\bin(m/2-1)=\frac{1}{2}(\bin(m-1)\setminus (1)).$$\end{corollary}

 \begin{remark}
Applying Lemma \ref{L_bin} repeatedly shows that if $m$ is odd and $k\geq 0$, then $2^k$ is not a part of $\bin(2^km-1)$. \end{remark}

We have the following non-recursive description of the partitions in $\SB(n)$. 

\begin{theorem}\label{thm_bin}For $n\geq 1$ we have $$\SB(n)=\{(\lambda\vdash n \ \big| \ \lambda\setminus (\lambda_1) \subseteq \bin(\lambda_1-1)\}. $$ 
\end{theorem}
\begin{proof}We prove the statement by induction. Since $\SB(1)=\{(1)\},\, \SB(2)=\{(2)\}, \,\SB(3)=\{(3), (2,1)\}, \,\SB(4)=\{(4)\}$, the statement holds for $n=1,2,3,4$. Let $m > 4$ and suppose the statement of the theorem is true if $n<m$. Let $\lambda\in \SB(m)$. \smallskip

\noindent \underline{Case 1:} $m=2t$, $t>2$. Then $\lambda=2\mu$ for some $\mu \in \SB(t)$. Using the induction hypothesis and Lemma \ref{L_bin}, we have \begin{align*}\lambda\setminus (\lambda_1)= 2(\mu\setminus(\mu_1))\subseteq 2\bin(\mu_1-1)& \subseteq 2\bin(\mu_1-1)\sqcup (1)\\ & =\bin(2\mu_1-1)=\bin(\lambda_1-1).\end{align*}\smallskip

\noindent \underline{Case 2:} $m=2t+1$, $t\geq 2$. Then $\lambda=2\mu\sqcup (1)$ for some $\mu\in \SB(t)$ or $\lambda=(2\eta_1-1)\sqcup 2(\eta\setminus(\eta_1))$ for some $\eta \in \SB(t+1)$.
\begin{itemize}
\item[(i)] In the first subcase, $\lambda_1 = 1$ and $$\lambda\setminus (\lambda_1) = 2(\mu\setminus (\mu_1)
\subseteq 2\bin(\mu_1-1)\sqcup (1) = \bin(\lambda_1-1). $$ 
 \smallskip

\item[(ii)] If $\lambda=(2\eta_1-1)\sqcup 2(\eta\setminus(\eta_1))$ for some $\eta \in \SB(t+1)$, we have $$\bin(\lambda_1-1)=\bin(2(\eta_1-1))=2\bin(\eta_1-1).$$ Then $$\lambda\setminus (\lambda_1)=2(\eta\setminus (\eta_1)) \subseteq 2\bin(\eta_1-1)=\bin(\lambda_1-1).$$
\end{itemize}

Conversely, let $\mu\vdash m$ be such that $\mu\setminus (\mu_1)\subseteq \bin(\mu_1-1)$.

\noindent \underline{Case I:} $m=2t$, $t\geq 2$. Note that $1\not \in \mu$ since otherwise $\mu_1$ is odd and $\bin(\mu_1-1)$ does not contain $1$ as a part. Thus all the parts of $\mu$ are even. Consider the partition $\nu=\mu/2$. Then $\nu\setminus(\nu_1)=(\mu\setminus(\mu_1))/2$. Since $1\not \in \mu$, we have $\mu\setminus (\mu_1)\subseteq \bin(\mu_1-1)\setminus (1)$. 
Using Corollary \ref{C_bin}, we get
$$\bin(\nu_1-1)=\bin(\mu_1/2-1)=\frac{1}{2}(\bin(\mu_1-1)\setminus (1))\supseteq (\mu\setminus(\mu_1))/2=\nu\setminus (\nu_1).$$ 
By the induction hypothesis $\nu\in \SB(t)$, so $\mu=2\nu \in \SB(m)$.\smallskip

\noindent \underline{Case II:} $m=2t+1$, $t\geq 2$. Then either (i) the smallest part of $\mu$ is $1$ and there are no other odd parts, or (ii) the largest part of $\mu$ is odd. The latter includes the case $\lambda=(m)$. 
\begin{itemize}
\item[(i)] In the first subcase, $\mu_1$ is even and $\mu\setminus(\mu_1, 1)\subseteq \bin(\mu_1-1)\setminus (1)$. Let $\nu=(\mu\setminus(1))/2$. 
As in Case I, using Corollary \ref{C_bin}, we get
$$\bin(\nu_1-1)=\bin(\mu_1/2-1)=\frac{1}{2}(\bin(\mu_1-1)\setminus (1))\supseteq(\mu\setminus(\mu_1, 1))/2=\nu\setminus(\nu_1).$$ 
 By the induction hypothesis, $\nu\in \SB(t)$ and hence $\mu=2\nu\sqcup(1)\in \SB(m)$.

\item[(ii)] In the second subcase, $\mu_1$ is odd and since $t\geq 2$, $\mu_1\geq 3$. If $\ell(\lambda)=1$, by property (c) above, $\lambda\in \SB(m)$. If $\ell(\lambda)\geq 2$, let $\nu=(\mu_1+1, \mu_2, \ldots, \mu_\ell)/2$. We have $$\bin(\nu_1-1)=\bin((\mu_1-1)/2)=\frac{1}{2}\bin(\mu_1-1)\supseteq (\mu\setminus (\mu_1))/2=\nu\setminus (\nu_1).$$ By the induction hypothesis, $\nu\in \SB(t+1)$. Hence $\mu=(2\nu_1-1, 2\nu_2, \ldots, 2\nu_\ell)\in \SB(m)$. 
\end{itemize}
\end{proof}

\begin{remark}
We define $\theta: \SB(n) \to \HB(n-1)$ by $\theta(\lambda) = \bin(\lambda_1-1) \sqcup \lambda\setminus (\lambda_1)$. By Theorem \ref{thm_bin}, $\theta(\lambda)$ is a binary partition whose parts have multiplicity at most 2. As $\bin$ preserves the sum of parts, $\theta(\lambda) \in \HB(n-1)$. For any partition $\mu$, let $\mu^d$ be the partition containing one copy of each part of $\mu$. Then $\theta^{-1}(\mu) = (1+|\mu^d|) \sqcup \mu \setminus \mu^d$, which is in $\SB(n)$, again by Theorem \ref{thm_bin}. 
\end{remark}

We now give a recursive definition of hyperbinary partitions that lets us find the total number of parts in $\HB(n)$ recursively.

\begin{definition}We define the set $\HB'(n)$ recursively by $\HB'(1)=\{(1)\}$, $\HB'(2)=\{(2),(1,1)\}$ , and for $n\geq 1$ \begin{align*}\HB'(2n)& = 2\HB'(n)\bigcup 2HB'(n-1)\sqcup(1,1),\\ \HB'(2n+1)& = 2\HB'(n)\sqcup(1).
\end{align*}\end{definition}

The union is disjoint in the definition of $\HB'(2n)$.

\begin{theorem} For $n\geq 1$, $|\HB'(n)|=\sb(n+1)$.
\end{theorem}

\begin{proof}
By inspection, $|\HB'(1)|=\sb(2)=1$ and $|\HB'(2)|=\sb(3)=2$. Let $n\geq 2$ and assume that for all $m<n$ we have $|\HB'(m)|=\sb(m+1)$. 
If $n=2t$ is even, using the recurrence in the definition of $\HB'(n)$ and the induction hypothesis, we have \begin{align*}|\HB'(n)|=|\HB'(2t)|& =|\HB'(t)|+|\HB'(t-1)|\\ & =\sb(t+1)+\sb(t)\\ & =\sb(2t+1)\\ & =\sb(n+1).\end{align*} Similarly, if $n=2t+1$ is odd, we have \begin{align*}|\HB'(n)|=|\HB'(2t+1)| =|\HB'(t)| =\sb(t+1) =\sb(2(t+1)) =\sb(n+1).\end{align*} 
\end{proof}

\begin{theorem}
For $n \geq 1$, $\HB'(n) = \HB(n)$.
\end{theorem}

\begin{proof}
The definition of $\HB'(n)$ implies that $\HB'(n) \subseteq \HB(n)$. Since $|\HB'(n)| = |\HB(n)|$, the result follows.
\end{proof}

\subsection{Parity Results for $\sb(n)$}

In \cite[Theorem 2]{A19} it is shown that $\sf(n)$ is even if and only if $n\equiv 0 \pmod 3$. The analogous result holds for $\sb(n)$. 

\begin{proposition}For $n\geq 0$, $\sb(n)$ is even if and only if $n\equiv 0 \pmod 3$. \end{proposition} 

\begin{proposition}
For $n\geq 0$,
$\sb(4n+1) \equiv \sb(n+1)$.
\end{proposition}

\subsection{Total Number of Parts in $\SB(n)$ and $\HB(n)$}
 
Let $\psb(n)$ be the total number of parts in all of the partitions in $\SB(n)$.

From the initial conditions and recurrences for $\SB(n)$,
the sequence $\psb(n)$ can be defined directly by $\psb(1)=\psb(2) = 1$, $\psb(3)=3$, and for $n\geq 2$, \begin{align*}\psb(2n)& =\psb(n),\\ \psb(2n+1)&=\psb(n) + \sb(n) + \psb(n+1).\end{align*}

The following table shows the values of  $\psb(n)$ for $n = 1, \ldots, 12$.

$$\begin{array}
{c|cccccccccccc}
n  & 1 & 2 & 3 & 4 & 5 & 6 & 7 & 8 & 9 & 10 & 11 & 12 \\
\hline
\psb(n) & 1 &  1 &  3 &  1 &  5 &  3 &  6 &  1 &  7 &  5 &  11 &  3
\end{array} $$

\begin{proposition}
For $n \geq 0$, $\psb(4n+1) \equiv \psb(n+1).$
\end{proposition}

 The recursive definition of $\HB'(n)$ gives the formula for $\phb(n)$, the total number of parts in all the hyperbinary partitions of $n$:  $\phb(1) =1$,  $\phb(2)=3$  and for $n\geq 2$, 
 \begin{align} 
 \phb(2n)& =\phb(n)+ \phb(n-1) + 2\sb(n),\notag \\ \phb(2n+1)& = \phb(n) + \sb(n+1).
\end{align}

The following table shows the total number of parts in the hyperbinary partitions of $n$, $\phb(n)$, for $n = 1, 2, \ldots, 12$.

$$\begin{array}
{c|ccccccccccccc}
n & 1 & 2 & 3 & 4 & 5 & 6 & 7 & 8 & 9 & 10 & 11 & 12 \\
\hline
\phb(n) & 1 & 3 & 2 & 6 & 5 & 9 & 3 & 10 & 9 & 17 & 7 & 18
\end{array} $$

\begin{proposition}
For $n \geq 1$, $\phb(2n-1) + \phb(2n+1) + \phb(4n+1) \equiv 0$.
\end{proposition}

\section{Summary of notation}\label{sec_notation}

For reference, here are the definitions of six classical sequences.

$$\begin{array}
{c|c|c|c} & \text{Fibonacci} & \text{Lucas} & \text{Pell} \\ n& f(n) & l(n) & p(n) \\ \hline 0 & 1& 2& 0\\ 1& 1& 1& 1\\ n\geq 2 & f(n-1)+f(n-2)& l(n-1) + l(n-2)& 2 p(n-1) + p(n-2)
\end{array} $$

$$\begin{array}
{c|c|c|c} & \text{tribonacci} & \text{Padovan} & \text{Narayana's cows} \\ n& t(n) & \pa(n) & \nc(n) \\ \hline 0 & 0& 1& 1\\ 1& 0& 0& 1\\ 2& 1& 0& 1\\ n\geq 3 & t(n-1)+t(n-2)& \pa(n-2)+\pa(n-3)& \nc(n-1)+\nc(n-3)
\\ & + \, t(n-3) & &
\end{array} $$

The next table summarizes the notation used for various sets and sequences in the article.

$$\begin{array}
{l|llll} 
 & \text{set of} & \text{self-similar} & \text{total number} & \text{corresponding} \\
 & \text{partitions} & \text{sequence} & \text{of parts} & \text{bijective set} \\
 \hline
\text{semi-Fibonacci} & \i4 \SF & \h5 \sf  & \h5 \psf & \h5\OB  \\
\text{semi-tribonacci} &\i4  \ST & \h5  \st & \h5 \pst & \h5 \OBST \\
\text{semi-Padovan} &\i4  \SPa & \h5 \spa & \h5 \pspa & \h5\OB_{\mathrm{R}} \\
\text{modified semi-Padovan} & \i4 \SPa' & \h5 \spa' & \h5 \pspa' & \h5\OB_{\mathrm{R}'} \\
\text{semi-Pell} & \i4 \SP & \h5 \sp & \h5 \psp & \h5\overline\OB' \\
\text{semi-Narayana's cows} & \i4 \SNc & \h5 \snc & \h5 \psnc & \h5\mathcal{OB}_{1,3} \\
\text{delayed semi-Fibonacci} & \i4 \DSF & \h5 \dsf & \h5 \pdsf & \h5\OB_1 \\
\text{semi-Lucas} &\i4  \SL & \h5 \sl & \h5 \psl & \h5\overline\OB^* \\
\text{Stern-Brocot} & \i4 \SB & \h5 \sb & \h5 \psb & \h5\HB \\
\end{array} $$

We summarize the definitions of sets of partitions and overpartitions used in this paper. 
\medskip

$$\begin{array}
{l|l} 
 \text{set} & \text{description} \\
 \hline
 \B & \text{binary partitions}  \\
 \\
\OB & \text{odd binary partitions}  \\
\OB_1 &  \text{largest part must be repeated} \\ \OB_2 & \text{first two parts consecutive powers of } 2\\ 
\OB_{\mathrm{R}} & \text{each part is repeated} \\
\OB_{\mathrm{R}'} & \text{each part is repeated except possibly   the largest} \\
\mathcal{OB}_{1,3} & \text{each part has multiplicity $1$ or $3$}\\
\\
\overline\OB & \text{odd binary overpartitions} \\
\overline\OB' & \text{largest part must be overlined} \\
\overline\OB^* & \text{for $n\geq 1$, only the largest part may be overlined and only if repeated} \\
\OBST & \text{largest part cannot be overlined and} \\ & \text{\ \ \ \ \ other parts  may be overlined only if repeated} \\
\end{array} $$

\section{Concluding remarks} \label{sec_conclusions}

We defined recursively seven families of sets  of partitions enumerated by self-similar sequences. In each case, we found  a non-recursive description of the partitions, usually in terms of $2$-adic valuations of parts. Using generating functions, we also found 
 subsets of odd binary partitions or overpartitions enumerated by the sequences. We then found bijections between the corresponding sets of partitions enumerated by the same sequence. 

It might be interesting to study other recursively defined sets of partitions, whether enumerated by self-similar sequences or not. 

The sub-sequences at odd indices $\sf(2n+1)$, $\st(2n+1)$, $\sl(2n+1)$, $\sp(2n+1)$, and  $\spa(2n+1)$ are increasing, although very slowly. It would be interesting to compare their asymptotic growth rates.

A natural generalization of the recursion $a(2n) = a(n)$ is $a(m n ) = a(n)$, leading to recursively defined sets of partitions enumerated by $a(n)$ as before. Alanazi, Munagi, and Nyirenda did this for semi-$m$-Fibonacci
partitions \cite{AMN}; the same could be done for the $m$-ary analogs of $\st(n)$, $\sl(n)$, and the other sequences presented here.

\end{document}